\documentclass[11pt]{amsart}

\usepackage{geometry}
\usepackage{amssymb}
\usepackage{mathtools}
\usepackage{xcolor}
\usepackage[colorlinks=true,citecolor=red,linkcolor=blue,urlcolor=red]{hyperref}

\numberwithin{equation}{section}

\DeclareMathOperator{\vol}{vol}
\DeclareMathOperator{\conv}{conv}
\DeclareMathOperator{\ord}{ord}
\DeclareMathOperator{\aff}{aff}
\DeclareMathOperator{\intr}{int}
\DeclareMathOperator{\GL}{GL}

\newtheorem{thm}{Theorem}[section]
\newtheorem{lem}[thm]{Lemma}
\newtheorem{prop}[thm]{Proposition}
\newtheorem{cor}[thm]{Corollary}

\newtheorem{conj}[thm]{Conjecture}

\theoremstyle{definition}

\newtheorem{rem}[thm]{Remark}

\begin{document}

\title{The equality case of Ehrhart's volume conjecture}

\author{Jihao Liu}
\address{Department of Mathematics, Peking University, No. 5 Yiheyuan Road, Haidian District, Beijing 100871, China}
\address{Beijing International Center for Mathematical Research, Peking University, No. 5 Yiheyuan Road, Haidian District, Beijing 100871, China}
\email{liujihao@math.pku.edu.cn}

\date{\today}

\begin{abstract}
We prove that every full-dimensional compact convex body in $\mathbb R^n$ whose barycenter is its unique interior lattice point and whose volume is $(n+1)^n/n!$ is a unimodular image of the simplex $(n+1)\Delta_n-(1,\dots,1)$. This resolves the equality case of Ehrhart's volume conjecture, as a counterpart of the inequality part recently proved by OpenAI. The main result of this paper is obtained by generative AI, particularly GPT-5.6-sol, Fable 5, and the Danus system.
\end{abstract}

\maketitle

\setcounter{tocdepth}{1}
\tableofcontents

\section{Introduction}\label{sec:introduction}

In 1964, Ehrhart \cite{Ehr64} proposed the following generalization of Minkowski's fundamental theorem (see, e.g., \cite[Section~5]{GL87}); we state it in the formulation given by Nill and Paffenholz \cite[Conjecture~1.1]{NP14}. Let $\Delta_n=\conv\{0,e_1,\dots,e_n\}\subset\mathbb R^n$ be the standard simplex.

\begin{conj}[Ehrhart's volume conjecture, {\cite{Ehr64}}, as formulated in {\cite[Conjecture~1.1]{NP14}}]\label{conj:ehrhart}
Let $K\subset\mathbb R^n$ be an $n$-dimensional convex body whose barycenter is the origin. If the origin is the only interior lattice point of $K$, then
\[
\vol(K)\leq\frac{(n+1)^n}{n!},
\]
where equality holds if and only if $K$ is unimodularly equivalent to $(n+1)\Delta_n$.
\end{conj}

Ehrhart proved the inequality in dimension $2$ \cite{Ehr55a} and for simplices in every dimension \cite{Ehr79} (see \cite{NP14}). For an arbitrary centered body $K$ as in Conjecture~\ref{conj:ehrhart}, the Milman--Pajor symmetrization inequality \cite[Corollary~3(2)]{MP00} together with Minkowski's theorem gives $\vol(K)\leq 4^n$ \cite[p.~312]{HHH16}, and this bound was subsequently improved to $4^ne^{-c\sqrt n}$ by thin-shell estimates \cite[Proposition~6.2]{HST+22}, and to $4^ne^{-cn}$ by combining the lower bound of \cite[Theorem~4.1]{CHM+24} on the volume of $K\cap(-K)$ in terms of the isotropic constant with the boundedness of the isotropic constant supplied by the resolution of Bourgain's slicing problem \cite[Theorem~1.2]{KL25}, where $c>0$ denotes a positive absolute constant, not necessarily the same at each occurrence; all of these bounds remain exponentially far from the conjectured one. Berman and Berndtsson proved the inequality for centered rational polytopes admitting a primitive facet presentation with all facet constants at most $1$, and for bodies in the positive orthant with barycenter $(1,\dots,1)$, by complex-analytic methods \cite[Corollary~1.4 and Theorem~1.5]{BB17}; a derivation of the latter case from Gr\"unbaum's inequality \cite[Theorem~2 and its proof]{Gru60} is recorded in \cite[Remark~3.2]{BB17}. Nill and Paffenholz proved the inequality for convex bodies contained in the polar of a lattice polytope and classified the equality case in that setting \cite[Theorem~1.4]{NP14}, and formulated the equality refinement for general convex bodies as \cite[Conjecture~1.1]{NP14}.

On August 1, 2026, OpenAI released a public report presenting ten advances in mathematics and theoretical computer science, obtained by an internal version of its Astra model \cite{OAI26}. The eighth result in the collection is the inequality part of Ehrhart's volume conjecture:

\begin{thm}[{\cite[Chapter~8, Theorem~1.1]{OAI26}}]\label{thm:oai}
Let $K\subset\mathbb R^n$ be a full-dimensional compact convex body with barycenter $0$. If $\intr(K)\cap\mathbb Z^n=\{0\}$, then
\[
\vol(K)\leq\frac{(n+1)^n}{n!}.
\]
\end{thm}

The strategy of \cite{OAI26} for Theorem~\ref{thm:oai} can be considered as a convex-geometry analogue of Fujita's inequality $(-K_X)^d\leq(d+1)^d$ for K-semistable Fano manifolds $X$ of dimension $d$ \cite{Fuj18}, and its streamline follows the structure of the analytic method of Berman and Berndtsson \cite{BB17}; both debts are recorded in \cite{OAI26} itself, which uses the same point filtration and vanishing-order count as the proof of \cite[Theorem~5.1]{Fuj18} (the filtration machinery is developed in \cite[Section~4]{Fuj18}) and the one-dimensional Bergman mechanism of \cite[Theorem~2.3]{BB17}. In this frame, the barycenter hypothesis supplies a real Monge--Amp\`ere potential \cite{BB13} playing the role of a K\"ahler--Einstein metric, and the unique interior lattice point makes the first weighted Bergman space one-dimensional, so that the convexity supplied by \cite{Ber06} acts on a scalar partition function; the point filtration, squeezed against a Schwarz-lemma estimate, then yields the sharp bound. In the rational-polytope setting of \cite{BB17}, Bergman convexity was combined with a Moser--Trudinger inequality at a torus-fixed point; \cite{OAI26} replaces that half of the argument by Fujita's vanishing-order count, and the rationality restriction disappears. Closely related finite-level jet counts and saturation arguments appear on the Fano side in the work of Zhang \cite{Zha22,Zha25}; Remark~\ref{rem:dictionary} collects these parallels.

The equality part of Ehrhart's volume conjecture, namely the assertion that equality holds only for unimodular images of the simplex under an affine automorphism of the lattice $\mathbb Z^n$, that is, an integer linear map of determinant $1$ or $-1$ followed by a translation by an integer vector, was, however, left unresolved in \cite{OAI26}. We remark that \cite{OAI26} referred to \cite[Conjecture~1.1]{NP14} for the equality refinement; according to \cite{NP14}, the uniqueness question was already suggested in \cite{Ehr64}, and it was proved in dimension $2$ in \cite{Ehr55b}.

On the other hand, Fujita \cite[Theorem~1.1]{Fuj18} proved the corresponding equality case in the smooth K\"ahler--Einstein setting: when a K\"ahler--Einstein Fano manifold satisfies $(-K_X)^d=(d+1)^d$, one must have $X\cong\mathbb P^d$. The equality characterization was subsequently extended to singular $\mathbb Q$-Fano varieties. More precisely, if $X$ is an $n$-dimensional K-semistable $\mathbb Q$-Fano variety and $(-K_X)^n=(n+1)^n$, then $X\cong\mathbb P^n$ \cite[Theorems~13(2) and~36]{Liu18}. Liu and Zhuang obtained the corresponding Ding-semistable maximal-volume characterization from their Seshadri-constant criterion for projective space \cite[Theorems~2 and~10]{LZ18}. It is therefore natural to ask whether Fujita's equality analysis can be translated to convex geometry in a similar way, so as to completely resolve the equality part of Ehrhart's volume conjecture. The main theorem of this paper is that this is indeed the case.

\begin{thm}\label{thm:main}
Let $n$ be a positive integer and let $K\subset\mathbb R^n$ be a full-dimensional compact convex body with barycenter $0$, such that $\intr(K)\cap\mathbb Z^n=\{0\}$ and
\[
\vol(K)=\frac{(n+1)^n}{n!}.
\]
Then $K=U\bigl((n+1)\Delta_n-(1,\dots,1)\bigr)$ for some $U\in\GL_n(\mathbb Z)$.
\end{thm}

\begin{rem}\label{rem:conventions}
Theorem~\ref{thm:main} resolves \cite[Conjecture~1.1]{NP14}, and hence the equality part of Conjecture~\ref{conj:ehrhart}, in every dimension. The two formulations agree: if $K=U(n+1)\Delta_n+t$ is a unimodular image of $(n+1)\Delta_n$ under $x\mapsto Ux+t$ with $U\in\GL_n(\mathbb Z)$ and $t\in\mathbb Z^n$, then the barycenter hypothesis forces $t=-U(1,\dots,1)\in\mathbb Z^n$, so that $K=U\bigl((n+1)\Delta_n-(1,\dots,1)\bigr)$; conversely every body of the latter form is a unimodular image of $(n+1)\Delta_n$. Together with Theorem~\ref{thm:oai}, Theorem~\ref{thm:main} completes Conjecture~\ref{conj:ehrhart}.
\end{rem}

\subsection*{Sketch of the proof}

The argument for Theorem~\ref{thm:main} has two halves, one analytic and one arithmetic. Write $S_n=(n+1)\Delta_n-(1,\dots,1)$ for the centered standard simplex. The frame of the analytic half is that of \cite{OAI26}, whose results we use directly rather than reprove: the barycenter hypothesis supplies the real Monge--Amp\`ere potential $\varphi$ of \cite{BB13} (see also \cite{CK15}), the spaces of Laurent series on the complex torus $(\mathbb C^*)^n$ that are square-integrable against $e^{-k\varphi}$ have orthogonal monomial bases indexed by the lattice points of $\intr(kK)$ \cite[Chapter~8, Lemma~2.1]{OAI26}, and the unique interior lattice point makes the Bergman-kernel convexity supplied by \cite{Ber06} act on scalar partition functions. On top of this frame, the proof introduces the following ingredients, which are new.

First, the volume hypothesis pins the vanishing-order filtrations of these spaces at \emph{every} point of the torus simultaneously: the normalized partition function of the filtered envelope ray based at $p$ is exactly linear, $L_p(t)=nt$, for every base point $p$ (Theorem~\ref{thm:gateway}). The proof squeezes an elementary jet-count lower bound, which consumes the entire volume hypothesis, against a Schwarz-lemma upper bound. A cognate single-point jet count, with the log-canonical-threshold ceiling in place of the Schwarz bound, appears on the Fano side in \cite[Proposition~3.1]{Zha25}; the logical direction there is reversed: a local stability hypothesis bounds the section count, while here the volume hypothesis forces the saturation, at every base point.

Second, the saturation has discrete consequences: for every $\varepsilon>0$ and all large $k$, no nonzero polynomial of total degree at most $(n+1-\varepsilon)k$ vanishes at all lattice points of $\intr(kK)$ (Theorem~\ref{thm:hypersurface}), and every lattice width of $K$ is at least $n+1$ (Corollary~\ref{cor:width}). For the simplex the degree threshold is attained exactly, at $(n+1)k-n$ (Proposition~\ref{prop:simplexdelta}).

Third, and this is the key new device, the base point is degenerated radially along a generic direction $v$: a Pl\"ucker weight-polytope argument turns the point-jet filtrations into monomial filtrations with integer weights, a single generic direction working for all $k$ simultaneously (Lemma~\ref{lem:grassmann}, Proposition~\ref{prop:weights}), and the associated toric rays subconverge to rays with exactly linear partition functions (Theorem~\ref{thm:limitray}). The equality case of the Pr\'ekopa--Leindler inequality, due to Dubuc \cite{Dub77} (see also \cite{KW22}), forces every such limit ray to be a translation ray (Proposition~\ref{prop:prekopa}); a translation ray, in turn, forces $K$ to be a pyramid whose cap volumes are exactly $s^n/n!$ and whose apex minimizes the chosen direction (Theorems~\ref{thm:pyramid} and~\ref{thm:apex}); here the empirical measure of exponent--weight pairs concentrates on the affine roof determined by the translation vector. Since the good directions are dense, a convexity lemma shows that a full-dimensional body with minimizing pyramid apices in a dense set of directions is a simplex (Theorems~\ref{thm:densapex} and~\ref{thm:simplex}). The volume and barycenter hypotheses then exhibit $K=AS_n$ with $|\det A|=1$ (Proposition~\ref{prop:reduction}). For concave transforms and asymptotic measures associated with multiplicative filtrations, see \cite{BC11}; for geodesic rays associated with analytic test configurations, see \cite{RW14}.

Fourth, the arithmetic half is a critical-lattice theorem which we believe to be of independent interest: $\mathbb Z^n$ is the \emph{only} lattice of determinant one whose nonzero points all avoid $\intr(S_n)$ (Theorem~\ref{thm:critical}). Its proof is short: a sign-flip argument shows that such a lattice contains no nonzero point of the open unit cube; Haj\'os's theorem \cite{Haj41}, resolving Minkowski's conjecture on cube tilings, then makes the lattice upper unitriangular after a permutation of coordinates; and an explicit shear-point construction rules out every nontrivial shear. The value of the critical determinant of $S_n$ is classical \cite{Ehr79}; the uniqueness of the critical lattice appears to be new. Applying the theorem to $A^{-1}\mathbb Z^n$ proves $A\in\GL_n(\mathbb Z)$ and hence Theorem~\ref{thm:main}.

We remark that the starting point of this paper is the observation that \cite{OAI26} transports the proof of Fujita's inequality, while the same paper of Fujita settles the Fano-side equality case. The core of the equality analysis of \cite{Fuj18} is the saturation of Seshadri constants: equality forces the Seshadri constants of $-K_X$ to equal $n+1$ at every point, and the extremal $\mathbb P^n$ attains this value exactly. Under the correspondence, the role of the Seshadri constant is played by the lattice width, and the counterpart of the pointwise Seshadri saturation is the width saturation of Corollary~\ref{cor:width}: every lattice width of an equality body is at least $n+1$, and the extremal simplex attains the bound, its minimal lattice width being exactly $n+1$. For polarized toric varieties, Ambro and Ito gave a quantitative comparison between the Seshadri constant at a general point and the minimal lattice width of the moment polytope \cite[Theorem~0.1]{AI20}; see also the corrigendum \cite{AI23}. The first half of the proof (Sections~\ref{sec:gateway} and~\ref{sec:discrete}) realizes exactly this plan, through the all-point saturation of Theorem~\ref{thm:gateway}. Theorem~\ref{thm:main} is, however, not obtained by translating the equality analysis of \cite{Fuj18} step by step: the mechanisms of the second half (the radial degeneration to monomial filtrations, the Pr\'ekopa--Leindler rigidity, and the critical-lattice theorem) have no counterparts in the K-stability argument, whose projective-geometric inputs, principally the geometry of rational curves, in turn have no convex-geometric analogues.

\begin{rem}[A dictionary]\label{rem:dictionary}
The correspondence behind this paper extends beyond the Seshadri--width pairing described above. The table below collects the parallels known to us, as a heuristic dictionary with pointers; the left column lives on Fano manifolds, the right column on convex bodies.

\begin{center}\small
\begin{tabular}{@{}p{0.46\textwidth}p{0.46\textwidth}@{}}
\hline
\emph{Fano side} & \emph{convex side}\\
\hline
K-semistable Fano $X$; $\vol(-K_X)$ \cite{BB17,OAI26} & centered $K$ with unique interior lattice point; $n!\,\vol(K)$\\[2pt]
sections of $-mK_X$ & lattice points of $\intr\bigl((m{+}1)K\bigr)$ (Lemma~\ref{lem:monomial})\\[2pt]
point filtration and vanishing-order count \cite[Section~4 and Theorem~5.1]{Fuj18} & jet filtrations $F^j_{k,p}$ (Section~\ref{sec:gateway})\\[2pt]
log-canonical-threshold ceiling for basis divisors \cite[Proposition~3.1]{Zha25} & Bergman--Schwarz slope bound $L_p(t)\leq nt$ (Lemmas~\ref{lem:berndtsson} and~\ref{lem:schwarz}); both are integrability thresholds\\[2pt]
basis-divisor vanishing orders; the invariant $S_L(E)$ \cite[Sections~2--3]{Zha22}, \cite[Section~3]{Zha25} & normalized jet sums $A_k$ (Theorem~\ref{thm:gateway}); cap-volume profile (Theorem~\ref{thm:pyramid})\\[2pt]
separation of jets at equality \cite[Proposition~3.2]{Zha25} & hypersurface exclusion and unisolvence (Theorem~\ref{thm:hypersurface}, Proposition~\ref{prop:simplexdelta})\\[2pt]
Seshadri-constant saturation $\varepsilon_p(-K_X)=n+1$ \cite[Theorem~2.3]{Fuj18} & lattice-width saturation $w_u(K)\geq n+1$ (Corollary~\ref{cor:width}); the Ambro--Ito comparison between Seshadri constants and minimal lattice width \cite[Theorem~0.1]{AI20} (see also the corrigendum \cite{AI23})\\[2pt]
infinitesimal Newton--Okounkov body and its cone rigidity \cite[Proposition~4.6]{Zha22} & roof function, cap homothety, and the pyramid (Theorem~\ref{thm:pyramid})\\[2pt]
Berndtsson convexity \cite[Theorem~1.1]{Ber06} & Pr\'ekopa--Leindler equality \cite[Th\'eor\`eme~12]{Dub77} (Proposition~\ref{prop:prekopa})\\[2pt]
characterizations of $\mathbb P^n$ (\cite[Theorem~1.1]{Fuj18}; \cite[Theorems~13(2) and~36]{Liu18}; \cite[Theorems~2 and~10]{LZ18}; \cite[Theorem~3.3]{Zha25}; \cite{LM25}) & no counterpart; replaced by dense apices and the critical lattice (Sections~\ref{sec:simplex} and~\ref{sec:critical})\\
\hline
\end{tabular}
\end{center}

The last row marks the boundary of the dictionary: every known Fano-side equality analysis exits through a characterization of $\mathbb P^n$, and none of these characterizations has a convex-geometric counterpart; conversely, the arithmetic endgame of Section~\ref{sec:critical} has no established complex-geometric avatar. It is an interesting question whether Theorem~\ref{thm:critical} corresponds to a rigidity property of toric degenerations of $\mathbb P^n$, and, in the opposite direction, whether the mechanisms of Sections~\ref{sec:rays}--\ref{sec:critical} admit complex-geometric counterparts under this dictionary.
\end{rem}

\medskip

\noindent\textbf{Structure of the paper.} Section~\ref{sec:preliminaries} collects the conventions, the transport potential of \cite{BB13}, the weighted Laurent spaces, and the elementary convexity and measure lemmas used throughout. Section~\ref{sec:gateway} proves the sharp jet-sum theorem: the filtered partition functions are exactly linear at every base point. Section~\ref{sec:discrete} deduces the hypersurface exclusion and the width bounds. Section~\ref{sec:rays} constructs the toric limit rays, applies the Pr\'ekopa--Leindler equality characterization, and proves the pyramid structure. Section~\ref{sec:simplex} proves that every equality body is a simplex. Section~\ref{sec:critical} proves the critical-lattice theorem for $S_n$, and Section~\ref{sec:proof} concludes.

\begin{rem}\label{rem:ai}
The main result of this paper was obtained using generative AI, particularly GPT-5.6-sol, Fable 5, and the Danus system. Danus is a specialized agent built on the Rethlas system and substantially more capable of conducting fundamental mathematical research. See \cite{Liu+26} and \cite{Ju+26} for detailed introductions to the Danus system and the Rethlas system, respectively.

The route to Theorem~\ref{thm:main} is, however, not completely automatic. The key human input from the author is the observation that the method of \cite{OAI26} essentially relies on the strategy of \cite{Fuj18}, whereas on the K-stability side the corresponding equality case (a K\"ahler--Einstein Fano manifold with $(-K_X)^d=(d+1)^d$ is isomorphic to $\mathbb P^d$) was proved in the same paper. Since the proof in \cite{OAI26} essentially relies on the ideas of the proof of \cite{Fuj18}, it was reasonable to expect that Theorem~\ref{thm:main} could be proved under a similar correspondence; in particular, the author suggested attacking the lattice-width saturation as the counterpart of the Seshadri-constant saturation in \cite{Fuj18}. This observation was given to the Danus system, which carried out the first half of the proof along the suggested correspondence, found the mechanisms of the second half, for which no correspondence was available, on its own, and produced the complete proof, internally verified by the system's proof-checking pipeline, in $3$ hours and $16$ minutes. Discussions with Fable 5 and GPT-5.6-sol-ultra also assisted the author in quickly grasping the methodology of the proof of \cite[Chapter~8, Theorem~1.1]{OAI26}. Human verification and polishing were done afterwards. We do not know whether the Danus system, at intermediate steps of its search, found further references related to the correspondences of Remark~\ref{rem:dictionary}, such as \cite{Zha22,Zha25}, and applied them in the proof presented here; the parallels with \cite{Zha22,Zha25} collected in Remark~\ref{rem:dictionary} were identified by the author only after the proof was completed. The references \cite{AI20,Per00} and their relation to the correspondences of Remark~\ref{rem:dictionary} were brought to the author's attention by Florin Ambro after the first version of this paper was posted. Due to the limitation of generative AI, it is possible that we have missed some related references in the literature, and we welcome any comments from experts.
\end{rem}

\subsection*{Acknowledgements}
The work was partially supported by the National Key R\&D Program of China \#\allowbreak 2024YFA1014400.
The author would like to thank other members of the Danus team (namely Guoxiong Gao, Zeming Sun, Bin Wu, Shurui Liu, Jiedong Jiang, Haocheng Ju, Leheng Chen, Ronnie Cheng, Xiping Zhang, and Bin Dong) and the Rethlas team (namely Haocheng Ju, Jiedong Jiang, Shurui Liu, Guoxiong Gao, Yuefeng Wang, Zeming Sun, Bin Wu, Liang Xiao, and Bin Dong) for their contributions to the development of Danus and Rethlas.
The author would like to thank Kewei Zhang for bringing the references \cite{Zha22,Zha25} to his attention and for useful discussions on the correspondences in Remark~\ref{rem:dictionary}.
The author would like to thank Florin Ambro for helpful comments on an earlier version and for bringing \cite{AI20,Per00} to his attention.
The author would like to thank Ruochuan Liu and Gang Tian for constant support and encouragement.

\section{Preliminaries}\label{sec:preliminaries}

In this section, we fix notation, collect the external inputs of the proof, and prove the elementary convexity and measure lemmas used throughout the paper. The Fano-side counterpart of this frame is the K\"ahler--Einstein and Bergman-kernel setting of \cite{BB13,BB17,OAI26}; see Remark~\ref{rem:dictionary}.

\subsection{Notation and conventions}\label{subsec:notation}

We work in $\mathbb R^n$ with the standard lattice $\mathbb Z^n$ and Lebesgue measure $\vol$; we write $e_1,\dots,e_n$ for the standard basis of $\mathbb R^n$, and $\langle\cdot,\cdot\rangle$ and $|\cdot|$ for the standard Euclidean pairing and norm. A \emph{convex body} $K\subset\mathbb R^n$ is a compact convex set; it is \emph{full-dimensional} if its interior $\intr(K)$ is nonempty. The \emph{barycenter} of a full-dimensional convex body $K$ is $\frac{1}{\vol(K)}\int_Ky\,dy$. The \emph{support function} of $K$ is $h_K(x)=\max_{y\in K}\langle y,x\rangle$. For a nonzero $u\in\mathbb Z^n$, the \emph{lattice width} of $K$ in direction $u$ is
\[
w_u(K)=\max_{x\in K}\langle u,x\rangle-\min_{x\in K}\langle u,x\rangle.
\]
Throughout the paper we write
\[
S_n=(n+1)\Delta_n-(1,\dots,1),\qquad\text{where }\Delta_n=\conv\{0,e_1,\dots,e_n\},
\]
for the centered standard simplex, and, for a full-dimensional compact convex body $K$ and a positive integer $k$,
\[
S_k(K)=\mathbb Z^n\cap\intr(kK),\qquad d_k=\#S_k(K),
\]
suppressing $K$ from the notation when it is clear from the context. No confusion with the simplex $S_n$ will arise: the dilation parameter is always denoted $k$, the subscript $n$ is reserved for the dimension, and when the body is the simplex itself we write $S_k(S_n)$.

\begin{lem}\label{lem:model}
The centered standard simplex satisfies
\[
S_n=\{x\in\mathbb R^n:x_i\geq-1\text{ for every }i\text{ and }x_1+\dots+x_n\leq1\},
\]
and $\intr(S_n)$ is described by the corresponding strict inequalities. Moreover, $S_n$ is full-dimensional with barycenter $0$, $\vol(S_n)=(n+1)^n/n!$, and $\intr(S_n)\cap\mathbb Z^n=\{0\}$.
\end{lem}

\begin{proof}
The vertices of $S_n$ are $s_0=-(1,\dots,1)$ and $s_i=-(1,\dots,1)+(n+1)e_i$ for $1\leq i\leq n$. Each vertex satisfies the listed inequalities, and each of the $n+1$ inequalities is an equality on exactly $n$ of the vertices, so the listed halfspace description is the facet description of $S_n$ and the interior is given by the strict inequalities. The barycenter of a simplex is the arithmetic mean of its vertices, and $\sum_{i=0}^ns_i=-(n+1)(1,\dots,1)+(n+1)(1,\dots,1)=0$. The volume is $(n+1)^n\vol(\Delta_n)=(n+1)^n/n!$. Finally, if $x\in\intr(S_n)\cap\mathbb Z^n$, then $x_i>-1$ for every $i$ and $\sum_ix_i<1$; integrality gives $x_i\geq0$ for every $i$ and $\sum_ix_i\leq0$, hence $x=0$.
\end{proof}

On the complex torus $X=(\mathbb C^*)^n$ we use the logarithmic coordinates
\[
z_i=\exp\Bigl(\frac{x_i}{2}+i\theta_i\Bigr),\qquad x\in\mathbb R^n,\ \theta\in(\mathbb R/2\pi\mathbb Z)^n,
\]
with the angular measure $d\theta$ the Haar (product Lebesgue) measure of $(\mathbb R/2\pi\mathbb Z)^n$ normalized to have total mass one, and we set $d\nu=dx\,d\theta$. For a Laurent monomial $z^m=z_1^{m_1}\cdots z_n^{m_n}$ we have $|z^m|^2=e^{\langle m,x\rangle}$. A function on $X$ is \emph{torus-invariant} if it depends only on $x$; we identify torus-invariant functions on $X$ with functions on $\mathbb R^n$ without further comment.

\subsection{The transport potential and the weighted Laurent spaces}\label{subsec:transport}

The following existence theorem of Berman and Berndtsson is the analytic starting point; it plays the role of a K\"ahler--Einstein metric on the toric side.

\begin{thm}[{\cite[Theorem~1.1]{BB13}}]\label{thm:bb13}
Let $K\subset\mathbb R^n$ be a full-dimensional compact convex body with $0\in\intr(K)$ whose barycenter is $0$. Then there exists a smooth strictly convex function $\varphi\colon\mathbb R^n\to\mathbb R$ such that $\nabla\varphi$ is a diffeomorphism from $\mathbb R^n$ onto $\intr(K)$, $\det(D^2\varphi)=e^{-\varphi}$, and $\varphi-h_K$ is bounded on $\mathbb R^n$.
\end{thm}

\begin{lem}\label{lem:potential}
Let $K\subset\mathbb R^n$ be a full-dimensional compact convex body whose barycenter is $0$. Then $0\in\intr(K)$, and for $\varphi$ as in Theorem~\ref{thm:bb13} we have
\[
(\nabla\varphi)_*\bigl(e^{-\varphi(x)}dx\bigr)=\mathbf 1_K(y)\,dy,\qquad\int_{\mathbb R^n}e^{-\varphi}\,dx=\vol(K).
\]
We call $\varphi$ a \emph{BB13 potential} of $K$. See also \cite[Theorem~2]{CK15} for the underlying transport problem.
\end{lem}

\begin{proof}
If $0\notin K$, strict separation would give a linear functional $\ell$ with $\ell>0$ on $K$, so the average of $\ell$ over $K$ would be positive, contradicting the barycenter hypothesis; thus $0\in K$. If $0$ were a boundary point of $K$, there would exist a supporting linear functional $\ell$ with $\ell\geq0$ on $K$ and $\ell(0)=0$. Since $K$ is full-dimensional, $\ell$ is strictly positive on a subset of $K$ of positive volume, so the average of $\ell$ over $K$ would be positive, contradicting the barycenter hypothesis. Thus $0\in\intr(K)$ and Theorem~\ref{thm:bb13} applies. For every bounded measurable function $a$ on $K$, the change of variables $y=\nabla\varphi(x)$ gives
\[
\int_{\mathbb R^n}a(\nabla\varphi(x))e^{-\varphi(x)}\,dx=\int_{\mathbb R^n}a(\nabla\varphi(x))\det(D^2\varphi(x))\,dx=\int_{\intr(K)}a(y)\,dy,
\]
using the Monge--Amp\`ere equation $\det(D^2\varphi)=e^{-\varphi}$. This is the pushforward identity, and taking $a\equiv1$ gives the integral formula.
\end{proof}

Fix a BB13 potential $\varphi$ of $K$ and regard it as a torus-invariant weight on $X$. For every positive integer $k$, let
\[
H_k=\Bigl\{f\text{ holomorphic on }X:\int_X|f|^2e^{-k\varphi}\,d\nu<\infty\Bigr\},
\]
a Hilbert space with the indicated weighted $L^2$ inner product, and let
\begin{equation}\label{eq:bergman}
B_k(z)=\sum_a|s_{k,a}(z)|^2,\qquad d\mu_k=\frac{B_ke^{-k\varphi}}{d_k}\,d\nu,\qquad d\mu=\frac{e^{-\varphi}}{\vol(K)}\,d\nu,
\end{equation}
where $(s_{k,a})_a$ is any orthonormal basis of $H_k$; the Bergman function $B_k$ does not depend on the choice of basis. Both $\mu_k$ and $\mu$ are probability measures: $\int_XB_ke^{-k\varphi}\,d\nu=\dim H_k$ by orthonormality, and $\dim H_k=d_k$ by Lemma~\ref{lem:monomial} below, while $\int_Xe^{-\varphi}\,d\nu=\vol(K)$ by Lemma~\ref{lem:potential}. By a \emph{weight} we mean a measurable real-valued function on $X$; the \emph{weighted holomorphic space} of a weight $w$ consists of the holomorphic functions $f$ on $X$ with $\int_X|f|^2e^{-w}\,d\nu<\infty$, a Hilbert space with the corresponding weighted $L^2$ inner product whenever $w$ is locally bounded above, as is every weight used in this paper. We use the following two results of \cite{OAI26} directly.

\begin{lem}[{\cite[Chapter~8, Lemma~2.1]{OAI26}}]\label{lem:monomial}
Let $K$ be a full-dimensional compact convex body with barycenter $0$ and let $\varphi$ be a BB13 potential of $K$. The Laurent monomials $z^m$ with $m\in S_k(K)$ form an orthogonal basis of $H_k$. In particular $H_k$ is finite-dimensional of dimension $d_k=\#S_k(K)$, and $d_k/k^n\to\vol(K)$ as $k\to\infty$. If moreover $\intr(K)\cap\mathbb Z^n=\{0\}$, then $H_1=\mathbb C$. More generally, the weighted holomorphic space defined by any weight $b$ with $b-\varphi$ bounded on $X$ equals $\mathbb C$ whenever $H_1=\mathbb C$; indeed the two weighted norms are equivalent, so the two spaces trivially coincide.
\end{lem}

\begin{lem}[{\cite[Chapter~8, Lemma~2.2]{OAI26}}]\label{lem:bergmanconv}
In the situation of Lemma~\ref{lem:monomial},
\[
\frac1k\log B_k\longrightarrow\varphi\quad\text{locally uniformly on }X,\qquad\|\mu_k-\mu\|_{\mathrm{TV}}\longrightarrow0,
\]
as $k\to\infty$, where $\|\cdot\|_{\mathrm{TV}}$ denotes the total-variation distance.
\end{lem}

We also need the following general form of the monomial-basis statement, in which the body need not have barycenter $0$ and the weight is only tied to the support function; it is used in Section~\ref{sec:rays}.

\begin{lem}[Support-growth basis lemma]\label{lem:basis}
Let $P\subset\mathbb R^n$ be a full-dimensional compact convex body and let $w$ be a torus-invariant weight on $X$ with $|w-h_P|$ bounded. Then the Laurent monomials $z^m$ with $m\in\mathbb Z^n\cap\intr(P)$ form an orthogonal basis of the weighted holomorphic space
\[
H(w)=\Bigl\{f\text{ holomorphic on }X:\int_X|f|^2e^{-w}\,d\nu<\infty\Bigr\},
\]
and $\int_Xe^{-w}\,d\nu$ is finite and positive whenever $0\in\intr(P)$. In particular, if $\intr(P)\cap\mathbb Z^n=\{0\}$, then $H(w)=\mathbb C$. Replacing $w$ by any weight $w'$ with $|w'-w|$ bounded changes neither $H(w)$ nor the finiteness of the integral.
\end{lem}

\begin{proof}
Since $|w-h_P|$ is bounded, the squared norm of $z^m$ is finite exactly when $I(m)=\int_{\mathbb R^n}e^{\langle m,x\rangle-h_P(x)}\,dx$ is finite. If $m\in\intr(P)$, choose $\delta>0$ with $m+\delta B\subseteq P$, where $B$ is the Euclidean unit ball; then $h_P(x)\geq\langle m,x\rangle+\delta|x|$, so $I(m)<\infty$. If $m\notin P$, separation gives a unit vector $v$ with $\langle m,v\rangle>h_P(v)$; the strict inequality persists on a spherical neighborhood of $v$, and polar coordinates show that the integrand grows exponentially on the corresponding cone, so $I(m)=\infty$. If $m\in\partial P$, choose a unit vector $v$ with $\langle m,v\rangle=h_P(v)$ and put $q(u)=h_P(u)-\langle m,u\rangle$ on the unit sphere, so that $q\geq0$, $q(v)=0$, and $q$ is Lipschitz, say $q(u)\leq A|u-v|$. In polar coordinates, on the spherical cap $\{|u-v|<1/\rho\}$ one has $\rho q(u)\leq A$, and the cap has surface measure at least a positive constant times $\rho^{-(n-1)}$ for large $\rho$; multiplying by the radial Jacobian $\rho^{n-1}$ shows that each radial slab of unit length contributes at least a fixed positive amount, so $I(m)=\infty$ (in dimension one the supporting ray gives the divergence directly). Now let $f$ be holomorphic on $X$ with Laurent expansion $f=\sum_mc_mz^m$. For fixed $x$, the angular Parseval identity and Tonelli's theorem give
\[
\int_X|f|^2e^{-w}\,d\nu=\sum_m|c_m|^2\int_{\mathbb R^n}e^{\langle m,x\rangle-w(x)}\,dx,
\]
with every term nonnegative; the integrability criterion forces $c_m=0$ unless $m\in\intr(P)$, and conversely every square-summable combination of the normalized monomials with $m\in\mathbb Z^n\cap\intr(P)$ lies in $H(w)$. This proves completeness, and angular integration proves orthogonality. If $0\in\intr(P)$, then $h_P(x)\geq\delta_0|x|$ for some $\delta_0>0$, so $\int e^{-w}\,d\nu<\infty$, and positivity is clear. The final assertion holds because a bounded change of weight produces equivalent norms.
\end{proof}

For $m\in S_k$ we write
\begin{equation}\label{eq:Ik}
\bigl\|z^m\bigr\|_k^2=I_k(m/k),\qquad I_k(u)=\int_{\mathbb R^n}e^{k(\langle u,x\rangle-\varphi(x))}\,dx,
\end{equation}
which follows from the angular Parseval identity; here $\|\cdot\|_k$ denotes the norm of $H_k$. We write $s_{k,m}=z^m/\sqrt{I_k(m/k)}$ for the normalized monomials. We will also need the precise Laplace asymptotic of the norms \eqref{eq:Ik}; since the statement of Lemma~\ref{lem:bergmanconv} does not record it, we prove it in Section~\ref{sec:gateway} (Proposition~\ref{prop:laplace}).

\subsection{Elementary convexity and measure lemmas}\label{subsec:elementary}

We collect five elementary lemmas. Throughout this subsection, $\alpha$ denotes a nonnegative real number; in the applications, $\alpha$ is instantiated as the dimension $n$.

\begin{lem}\label{lem:pinch}
Let $L\colon[0,\infty)\to\mathbb R$ be convex with $L(0)=0$. Assume that $L(t)\leq\alpha t$ for every $t\geq0$, that the right derivative $L'_+(0)$ exists, and that $L'_+(0)\geq\alpha$. Then $L(t)=\alpha t$ for every $t\geq0$ and $L'_+(0)=\alpha$. Moreover, if $(a_k)_k$ is a real sequence with $\liminf_ka_k\geq\alpha$ and $\limsup_ka_k\leq L'_+(0)$, then $a_k\to\alpha$.
\end{lem}

\begin{proof}
For a convex function on $[0,\infty)$ with $L(0)=0$, every secant slope from the origin is at least the right derivative at the origin: $L(t)/t\geq L'_+(0)\geq\alpha$ for $t>0$. The assumed upper bound $L(t)\leq\alpha t$ gives $L(t)=\alpha t$ for every $t\geq0$, and hence $L'_+(0)=\alpha$. If $\liminf a_k\geq\alpha$ and $\limsup a_k\leq L'_+(0)=\alpha$, then $a_k\to\alpha$.
\end{proof}

\begin{lem}\label{lem:secant}
For every positive integer $k$ let $u_k\colon[0,\infty)\to\mathbb R$ be convex with right derivative $g_k=u'_{k,+}(0)$, and suppose that $u_k(0)$ converges to a real number $\varphi_0$. Let $\psi\colon[0,\infty)\to\mathbb R$ satisfy $\psi(0)=\varphi_0$ and $\limsup_ku_k(t)\leq\psi(t)$ for every $t>0$, and suppose that $g=\lim_{t\downarrow0}(\psi(t)-\varphi_0)/t$ exists. Then $\limsup_kg_k\leq g$.
\end{lem}

\begin{proof}
Fix $t>0$. Convexity gives $g_k\leq(u_k(t)-u_k(0))/t$. Taking upper limits in $k$ and using $u_k(0)\to\varphi_0$ and $\limsup_ku_k(t)\leq\psi(t)$ yields $\limsup_kg_k\leq(\psi(t)-\varphi_0)/t$. Letting $t\downarrow0$ proves the lemma.
\end{proof}

\begin{lem}\label{lem:derivative}
Let $\mu$ be a probability measure on a measurable space $Y$, let $c\geq0$ be finite, and for $t\geq0$ let $\psi_t$ and $\varphi_0$ be measurable real functions on $Y$ with $\psi_0=\varphi_0$ and $0\leq(\psi_t-\varphi_0)/t\leq c$ for every $t>0$. Suppose that $g(y)=\lim_{t\downarrow0}(\psi_t(y)-\varphi_0(y))/t$ exists for $\mu$-almost every $y$. Define
\[
Z(t)=\int_Ye^{-(\psi_t-\varphi_0)}\,d\mu,\qquad L(t)=-\log Z(t).
\]
Then $L'_+(0)$ exists and equals $\int_Yg\,d\mu$.
\end{lem}

\begin{proof}
For $t>0$ put $h_t=(\psi_t-\varphi_0)/t$, so that $0\leq h_t\leq c$ and $h_t\to g$ pointwise $\mu$-almost everywhere. Since $0\leq1-e^{-y}\leq y$ for $y\geq0$, we have
\[
0\leq\frac{1-e^{-th_t}}{t}\leq h_t\leq c,\qquad\frac{1-e^{-th_t(y)}}{t}\longrightarrow g(y)\quad\text{for $\mu$-a.e.\ }y.
\]
Dominated convergence gives $(1-Z(t))/t\to\int_Yg\,d\mu$ as $t\downarrow0$. Since $Z(t)\to Z(0)=1$ and $-\log$ has derivative $-1$ at $1$, we conclude $L(t)/t\to\int_Yg\,d\mu$, which is the claim.
\end{proof}

\begin{lem}\label{lem:fatou}
Let $\mu_k$ and $\mu$ be probability measures on the same measurable space with $\|\mu_k-\mu\|_{\mathrm{TV}}\to0$, let $c$ be finite, and let $g_k,g$ be measurable functions with $0\leq g_k\leq c$, $0\leq g\leq c$, and $\limsup_kg_k(y)\leq g(y)$ for $\mu$-almost every $y$. Then
\[
\limsup_{k\to\infty}\int g_k\,d\mu_k\leq\int g\,d\mu.
\]
\end{lem}

\begin{proof}
Since $0\leq g_k\leq c$, we have $|\int g_k\,d\mu_k-\int g_k\,d\mu|\leq c\|\mu_k-\mu\|_{\mathrm{TV}}\to0$, so it suffices to bound $\limsup_k\int g_k\,d\mu$. Applying Fatou's lemma to the nonnegative functions $c-g_k$ gives
\[
\liminf_k\int(c-g_k)\,d\mu\geq\int\liminf_k(c-g_k)\,d\mu=\int\Bigl(c-\limsup_kg_k\Bigr)d\mu\geq\int(c-g)\,d\mu,
\]
and subtracting from $c$ proves the claim.
\end{proof}

\begin{prop}\label{prop:bridge}
Let $\mu_k,\mu,c,g_k,g$ be as in Lemma~\ref{lem:fatou} and put $A_k=\int g_k\,d\mu_k$. Let $L\colon[0,\infty)\to\mathbb R$ be a finite convex function with $L(0)=0$, $L'_+(0)=\int g\,d\mu$, and $L(t)\leq\alpha t$ for every $t\geq0$. Then $\limsup_kA_k\leq\alpha$. If in addition $\liminf_kA_k\geq\alpha$, then $A_k\to\alpha$, $L'_+(0)=\alpha$, and $L(t)=\alpha t$ for every $t\geq0$.
\end{prop}

\begin{proof}
Lemma~\ref{lem:fatou} gives $\limsup_kA_k\leq\int g\,d\mu=L'_+(0)$. For every $t>0$, convexity gives $L'_+(0)\leq L(t)/t\leq\alpha$, proving $\limsup_kA_k\leq\alpha$. The remaining assertions follow from Lemma~\ref{lem:pinch} applied to $L$ and $(A_k)_k$.
\end{proof}

\section{The sharp jet-sum theorem}\label{sec:gateway}

In this section, we prove the first main step of the analytic half: for a body satisfying the hypotheses of Theorem~\ref{thm:main}, the truncated vanishing-order filtrations of the spaces $H_k$ have sharp normalized dimension sums at every point of the torus, and the associated envelope partition functions are exactly linear (Theorem~\ref{thm:gateway}). On the Fano side, the counterpart of this section is the single-point jet count of \cite[Propositions~3.1 and~3.2]{Zha25}, which descends from \cite[Theorem~2.3]{Fuj18}; see Remark~\ref{rem:dictionary}.

Throughout this section, $K\subset\mathbb R^n$ is a full-dimensional compact convex body with barycenter $0$ and $\intr(K)\cap\mathbb Z^n=\{0\}$, the function $\varphi$ is a BB13 potential of $K$ (Lemma~\ref{lem:potential}), and $H_k$, $B_k$, $\mu_k$, $\mu$, $S_k$, $d_k$ are as in Section~\ref{subsec:transport}. For $p\in X$ and a holomorphic function $s$ on $X$, we write $\ord_p(s)$ for the vanishing order of $s$ at $p$, with the convention $\ord_p(0)=+\infty$, and we set
\[
F_{k,p}^j=\{s\in H_k:\ord_p(s)\geq j\},\qquad j\geq1,
\]
a decreasing filtration of $H_k$ by linear subspaces. We put
\[
N_k=(n+1)k,\qquad A_{k,p}=\frac{1}{kd_k}\sum_{j=1}^{N_k}\dim F_{k,p}^j.
\]

\subsection{The jet--interpolation identity}\label{subsec:jet}

The following jet--interpolation identity is the Laurent-monomial form of \cite[Proposition~1.1]{Per00}; see also \cite[Remark~2.3]{Per00} for its projective toric formulation. Our filtration index $j$ corresponds to jets through order $j-1$.

\begin{lem}\label{lem:jet}
Let $S\subset\mathbb Z^n$ be a finite set, let $H_S$ be the complex vector space spanned by the Laurent monomials $z^m$ with $m\in S$, and let $p\in X$ be any point. For a positive integer $j$, let $F_p^j\subset H_S$ be the subspace of functions vanishing to order at least $j$ at $p$, let $R_{<j}(S)$ be the space of restrictions to $S$ of complex polynomials in $n$ variables of total degree less than $j$, and let $I_{<j}(S)$ be the space of such polynomials vanishing at every point of $S$. Then
\[
\operatorname{codim}_{H_S}F_p^j=\dim R_{<j}(S)=\binom{n+j-1}{n}-\dim I_{<j}(S).
\]
In particular, the codimension does not depend on $p$, and
\[
\dim F_p^j\geq\#S-\binom{n+j-1}{n},
\]
with equality if and only if no nonzero polynomial of total degree less than $j$ vanishes on $S$.
\end{lem}

\begin{proof}
Choose local holomorphic coordinates $w$ near $p$ by $z_i=p_ie^{w_i}$, so that $w=0$ at $p$. For $m\in S$,
\[
z^m=p^me^{\langle m,w\rangle},\qquad\text{hence}\qquad\partial_w^\alpha\bigl(z^m\bigr)\Big|_{w=0}=p^mm^\alpha
\]
for every multi-index $\alpha$, where $p^m=\prod_ip_i^{m_i}\neq0$ and $m^\alpha=\prod_im_i^{\alpha_i}$. The order-less-than-$j$ jet of a holomorphic function at $p$ is represented by the derivatives $\partial_w^\alpha$ with $|\alpha|<j$, in any local holomorphic coordinate system. Thus, in the monomial basis of $H_S$, the jet map $H_S\to\mathbb C^{\{|\alpha|<j\}}$ has matrix $(p^mm^\alpha)_{\alpha,m}$, which is the matrix $(m^\alpha)_{\alpha,m}$ multiplied on the right by the invertible diagonal matrix $\operatorname{diag}(p^m)$. Its kernel is $F_p^j$, so
\[
\operatorname{codim}_{H_S}F_p^j=\operatorname{rank}(m^\alpha)_{|\alpha|<j,\,m\in S}.
\]
The row span of $(m^\alpha)_{|\alpha|<j}$, viewed inside the space of functions on $S$, is exactly the space of restrictions to $S$ of polynomials of total degree less than $j$, because the monomials $x^\alpha$ with $|\alpha|<j$ form a basis of that polynomial space. Hence the rank equals $\dim R_{<j}(S)$. Rank--nullity for the evaluation map from the $\binom{n+j-1}{n}$-dimensional space of polynomials of degree less than $j$ to functions on $S$ gives $\dim R_{<j}(S)=\binom{n+j-1}{n}-\dim I_{<j}(S)$, and the final assertions follow.
\end{proof}

Applying Lemma~\ref{lem:jet} to $S=S_k$ and $H_S=H_k$ (Lemma~\ref{lem:monomial}), we obtain, for every $p\in X$ and $j\geq1$,
\begin{equation}\label{eq:jetcount}
\dim F_{k,p}^j=d_k-r_{k,j},\qquad r_{k,j}:=\dim R_{<j}(S_k),
\end{equation}
so that $A_{k,p}$ is independent of $p$; we henceforth write $A_k$ for it.

\subsection{The Laplace asymptotic}\label{subsec:laplace}

The following per-monomial asymptotic refines Lemma~\ref{lem:bergmanconv}; it is used in Section~\ref{sec:rays}. Recall the integrals $I_k(u)$ from \eqref{eq:Ik}, and let $\varphi^*(u)=\sup_x(\langle u,x\rangle-\varphi(x))$ denote the Legendre transform of $\varphi$.

\begin{prop}\label{prop:laplace}
For every compact set $U\subset\intr(K)$, uniformly for $u\in U$,
\[
I_k(u)=e^{k\varphi^*(u)}\Bigl(\frac{2\pi}{k}\Bigr)^{n/2}\bigl(\det D^2\varphi(x_u)\bigr)^{-1/2}(1+o(1)),\qquad k\to\infty,
\]
where $x_u=(\nabla\varphi)^{-1}(u)$. Consequently, if $m_k\in S_k$ and $m_k/k\to y\in\intr(K)$, then
\[
\frac1k\log\bigl|s_{k,m_k}(x)\bigr|^2\longrightarrow\langle y,x\rangle-\varphi^*(y)
\]
locally uniformly in $x\in\mathbb R^n$.
\end{prop}

\begin{proof}
Write $H_u=D^2\varphi(x_u)$ and define the Bregman divergence
\[
D_u(v)=\varphi(v)-\varphi(x_u)-\langle u,v-x_u\rangle\geq0,
\]
which vanishes only at $v=x_u$, so that $I_k(u)=e^{k\varphi^*(u)}\int_{\mathbb R^n}e^{-kD_u(v)}\,dv$. The points $x_u$, $u\in U$, lie in a compact set $\mathcal C$ because $(\nabla\varphi)^{-1}$ is continuous. Fix a compact neighborhood $\mathcal C_1$ of $\mathcal C$. By positivity and continuity of $D^2\varphi$ there are $\lambda,L>0$ with
\[
\frac{\lambda}{2}|v-x_u|^2\leq D_u(v)\leq\frac{L}{2}|v-x_u|^2
\]
for all $u\in U$ and all $v$ in a fixed neighborhood of $x_u$ contained in $\mathcal C_1$, and Taylor's theorem gives
\[
kD_u\Bigl(x_u+\frac{r}{\sqrt k}\Bigr)\longrightarrow\frac12\langle H_ur,r\rangle
\]
uniformly for $u\in U$ and $r$ in every fixed bounded set, with the quadratic lower bound supplying an integrable Gaussian majorant there. It remains to bound the complementary region uniformly. Since $U$ has positive distance $\delta$ from $\partial K$, we have $h_K(v)-\langle u,v\rangle\geq\delta|v|$ for $u\in U$, and since $\varphi-h_K$ is bounded and $\varphi^*$ is bounded on $U$, there is $C_0$ with
\[
\langle u,v\rangle-\varphi(v)-\varphi^*(u)\leq-\delta|v|+C_0.
\]
On the compact region outside a fixed small neighborhood of the points $x_u$ and inside a large ball, strict positivity and compactness give $D_u\geq\eta>0$ uniformly in $u$, contributing $O(e^{-\eta k})$; outside the large ball, whose radius $R$ we choose with $\delta R\geq C_0+1$, the displayed bound makes the integrand at most $e^{-k(\delta|v|-C_0)}$, whose integral over $\{|v|>R\}$ is at most a constant multiple of $e^{-k}$ for all large $k$. After the change of variables $v=x_u+r/\sqrt k$, dominated convergence therefore applies uniformly in $u\in U$ and gives
\[
k^{n/2}e^{-k\varphi^*(u)}I_k(u)\longrightarrow\int_{\mathbb R^n}e^{-\langle H_ur,r\rangle/2}\,dr=(2\pi)^{n/2}(\det H_u)^{-1/2},
\]
which is the first assertion. For the second, $|s_{k,m_k}(x)|^2=e^{\langle m_k,x\rangle}/I_k(m_k/k)$ by \eqref{eq:Ik}, so
\[
\frac1k\log|s_{k,m_k}(x)|^2=\Bigl\langle\frac{m_k}{k},x\Bigr\rangle-\frac1k\log I_k(m_k/k)\longrightarrow\langle y,x\rangle-\varphi^*(y),
\]
locally uniformly in $x$, using the first assertion on a compact neighborhood $U$ of $y$ in $\intr(K)$ and the continuity of $\varphi^*$ on $U$.
\end{proof}

\subsection{Envelope rays and their convexity}\label{subsec:envelope}

Fix $p\in X$. For each $k$, choose an orthonormal basis $(s_{k,a})_a$ of $H_k$ adapted to the decreasing filtration $(F_{k,p}^j)_j$, meaning that every $F_{k,p}^j$ is the span of a subset of the basis; such bases exist: setting $F_{k,p}^0=H_k$ and noting that $F_{k,p}^j=0$ for all large $j$, one concatenates orthonormal bases of the successive orthogonal complements of $F_{k,p}^{j+1}$ in $F_{k,p}^j$ over $j\geq0$. Write $j_{k,a}=\ord_p(s_{k,a})$, truncate the orders at $N_k$ by setting $q_{k,a}=\min(j_{k,a},N_k)$, and define
\begin{equation}\label{eq:finiteray}
u_{k,p,t}(z)=\frac1k\log\sum_a|s_{k,a}(z)|^2e^{tq_{k,a}},\qquad t\geq0,
\end{equation}
so that $u_{k,p,0}=\frac1k\log B_k$. For an adapted basis, $\#\{a:j_{k,a}\geq j\}=\dim F_{k,p}^j$, and the telescoping identity $e^{tq_{k,a}}=1+\sum_{j=1}^{N_k}(e^{tj}-e^{t(j-1)})\mathbf 1_{j_{k,a}\geq j}$ writes $\sum_a|s_{k,a}(z)|^2e^{tq_{k,a}}=B_k(z)+\sum_{j=1}^{N_k}(e^{tj}-e^{t(j-1)})K_{k,p}^j(z)$, where $K_{k,p}^j$ is the diagonal reproducing kernel of the subspace $F_{k,p}^j$; hence \eqref{eq:finiteray} does not depend on the choice of adapted basis. The initial slope
\begin{equation}\label{eq:slope}
g_{k,p}(z)=\partial_tu_{k,p,t}(z)\Big|_{t=0^+}=\frac{\sum_a(q_{k,a}/k)|s_{k,a}(z)|^2}{B_k(z)}
\end{equation}
takes values in $[0,n+1]$.

\begin{lem}\label{lem:envelope}
For each $M\geq1$, let $\Psi_M$ be the upper-semicontinuous regularization of $\sup_{k\geq M}U_k$ on $X\times\mathbb C^*$, where $U_k(z,\tau)=\frac1k\log\sum_a|s_{k,a}(z)\tau^{q_{k,a}}|^2$, and let $\Psi=\lim_{M\to\infty}\Psi_M$. Then $\Psi$ is plurisubharmonic on $X\times\mathbb C^*$ and invariant under rotations of $\tau$, and its radial slices $\psi_{p,t}(z)=\Psi(z,e^{t/2})$ satisfy
\[
\psi_{p,0}=\varphi,\qquad0\leq\psi_{p,t}-\varphi\leq(n+1)t\quad(t\geq0).
\]
For each $z$, the function $t\mapsto\psi_{p,t}(z)$ is convex, so the initial velocity
\[
g_p(z)=\lim_{t\downarrow0}\frac{\psi_{p,t}(z)-\varphi(z)}{t}
\]
exists and lies in $[0,n+1]$, and $\limsup_kg_{k,p}(z)\leq g_p(z)$ for every $z$.
\end{lem}

\begin{proof}
Each $U_k$ is $\frac1k$ times the logarithm of a finite sum of squared absolute values of the holomorphic functions $(z,\tau)\mapsto s_{k,a}(z)\tau^{q_{k,a}}$, hence plurisubharmonic. Since $0\leq q_{k,a}\leq N_k=(n+1)k$, we have
\begin{equation}\label{eq:lipschitz}
u_{k,p,0}\leq u_{k,p,t}\leq u_{k,p,0}+(n+1)t\qquad(t\geq0),
\end{equation}
and, together with the local uniform convergence $u_{k,p,0}\to\varphi$ (Lemma~\ref{lem:bergmanconv}), this makes the family $\{U_k:k\geq M\}$ locally uniformly bounded above on $X\times\mathbb C^*$; the two-sided bound $|U_k(z,\tau)-U_k(z,1)|\leq(n+1)\bigl|\log|\tau|^2\bigr|$ covers parameter annuli. The upper-semicontinuous regularization of a locally upper-bounded supremum of plurisubharmonic functions is plurisubharmonic, so each $\Psi_M$ is plurisubharmonic; the sequence $\Psi_M$ decreases in $M$, and its limit $\Psi$ is plurisubharmonic and finite by the same local bounds. Every $U_k$ depends on $\tau$ only through $|\tau|$, hence so does $\Psi$. At $\tau=1$, local uniform convergence of $u_{k,p,0}$ to the continuous function $\varphi$ gives $\Psi(z,1)=\varphi(z)$: the annular Lipschitz bound $|U_k(z,\tau)-U_k(z,1)|\leq(n+1)\bigl|\log|\tau|^2\bigr|$ controls the tail suprema near $\tau=1$ by their values at $\tau=1$, so the upper-semicontinuous regularization adds nothing on that slice. That is, $\psi_{p,0}=\varphi$; passing to regularized tail suprema in \eqref{eq:lipschitz} gives the two-sided slice bounds. The restriction of the plurisubharmonic, rotation-invariant $\Psi$ to a fixed $z$ is a radial subharmonic function of $\tau$, hence convex in $t=\log|\tau|^2$; the slice bounds place its initial velocity $g_p(z)$ in $[0,n+1]$. Finally, for fixed $z$ the functions $t\mapsto u_{k,p,t}(z)$ are convex with $u_{k,p,0}(z)\to\varphi(z)$ and $\limsup_ku_{k,p,t}(z)\leq\psi_{p,t}(z)$ by the envelope construction, so Lemma~\ref{lem:secant} gives $\limsup_kg_{k,p}(z)\leq g_p(z)$.
\end{proof}

\begin{lem}\label{lem:berndtsson}
Let $w_0$ be a torus-invariant weight on $X$ whose weighted holomorphic space consists precisely of the constant functions and with $V_0=\int_Xe^{-w_0}\,d\nu\in(0,\infty)$. Let $\Psi$ be a plurisubharmonic function on $X\times\{\tau\in\mathbb C:|\tau|>1\}$, invariant under rotations of $\tau$, whose slices $\psi_t(z)=\Psi(z,e^{t/2})$ satisfy $w_0\leq\psi_t\leq w_0+ct$ for a constant $c$ and all $t>0$, and suppose $\psi_t\to\psi_0=w_0$ pointwise as $t\downarrow0$. Define
\[
Z(t)=\frac{1}{V_0}\int_Xe^{-\psi_t}\,d\nu,\qquad L(t)=-\log Z(t)\qquad(t\geq0).
\]
Then $L$ is finite and convex on $[0,\infty)$ with $L(0)=0$ and $0\leq L(t)\leq ct$. If moreover the pointwise initial velocity $g(z)=\lim_{t\downarrow0}(\psi_t(z)-w_0(z))/t$ exists for $\mu_0$-almost every $z$, where $d\mu_0=V_0^{-1}e^{-w_0}\,d\nu$, then $L'_+(0)=\int_Xg\,d\mu_0$.
\end{lem}

\begin{proof}
The slice bounds give $e^{-ct}V_0\leq\int_Xe^{-\psi_t}\,d\nu\leq V_0$, so $Z(t)\in[e^{-ct},1]$ and $0\leq L(t)\leq ct$; also $Z(0)=1$, so $L(0)=0$. Since $\psi_t-w_0$ is bounded for each fixed $t$, the weighted holomorphic space on the slice with weight $\psi_t$ coincides with that of $w_0$, hence consists precisely of the constants, by the equivalence of the two weighted norms. Write $d\nu=\pi^{-n}\prod_i|z_i|^{-2}\,d\lambda(z)$, where $d\lambda$ is Lebesgue measure on $\mathbb C^n\supset X$. The function $\rho(z)=\sum_i\log|z_i|^2+n\log\pi$ is pluriharmonic on $X$, so $\Phi=\Psi+\rho$ is plurisubharmonic on the pseudoconvex domain $X\times\{|\tau|>1\}$ and $e^{-\Phi}\,d\lambda=e^{-\Psi}\,d\nu$ on the slices. By \cite[Theorem~1.1]{Ber06}, the logarithm of the diagonal weighted Bergman kernel of the slices is plurisubharmonic in all variables or identically $-\infty$. Because each slice space is $\mathbb C$, spanned by the constant $1$ of squared norm $\int_Xe^{-\psi_t}\,d\nu\in(0,\infty)$, the slice Bergman kernel is the reciprocal of that norm; in particular it is finite and positive, the degenerate branch is excluded, and its logarithm equals $L(t)-\log V_0$, a function of $t=\log|\tau|^2$ alone. A finite rotation-invariant subharmonic function of $\tau$ is convex in $\log|\tau|^2$, so $L$ is convex on $(0,\infty)$; the bounds $0\leq L(t)\leq ct$ give continuity at $0$ and extend convexity to $[0,\infty)$. The final assertion is Lemma~\ref{lem:derivative} applied with $Y=X$ and the measure $\mu_0$, using $0\leq(\psi_t-w_0)/t\leq c$ and $Z(t)=\int_Xe^{-(\psi_t-w_0)}\,d\mu_0$.
\end{proof}

\subsection{The Schwarz bound}\label{subsec:schwarz}

\begin{lem}[Ball bound]\label{lem:ballbound}
Let $p\in X$, choose local holomorphic coordinates $\zeta$ near $p$ by $z_i=p_ie^{\zeta_i}$ on the coordinate ball $B_{2r}=\{|\zeta|<2r\}$ for some $r>0$, and put $M_r=\sup_{B_{2r}}\varphi$. In these coordinates, $d\nu=2^n(2\pi)^{-n}\,d\lambda(\zeta)$ is a constant multiple of Lebesgue measure. Moreover, there is a constant $C_r$, independent of $k$, with the following property: for every $k$, every orthonormal basis $(s_{k,a})_a$ of $H_k$ adapted to the vanishing-order filtration at $p$, every family of integers $0\leq q_{k,a}\leq\ord_p(s_{k,a})$, and every $\tau$ with $|\tau|\geq1$,
\[
\frac1k\log\sum_a|s_{k,a}(\zeta)\tau^{q_{k,a}}|^2\leq M_r+\frac1k\log\bigl(C_r^2d_k\bigr)\qquad\text{on the set }\{|\zeta|<r|\tau|^{-1}\}.
\]
In particular, for $\tau=e^{t/2}$ with $t\geq0$, the rays \eqref{eq:finiteray} satisfy $u_{k,p,t}\leq M_r+\frac1k\log(C_r^2d_k)$ on $\{|\zeta|<re^{-t/2}\}$.
\end{lem}

\begin{proof}
The measure identity follows from $x_i=\log|p_i|^2+2\operatorname{Re}\zeta_i$ and $\theta_i=\arg p_i+\operatorname{Im}\zeta_i$. If $s\in H_k$ has norm one, then $\int_{B_{2r}}|s|^2\,d\nu\leq e^{kM_r}$, so the holomorphic submean inequality gives $\sup_{B_r}|s|\leq C_re^{kM_r/2}$ with $C_r$ independent of $k$ and $s$. If $s$ vanishes to order $j$ at $p$, applying the one-variable Schwarz lemma on each complex line through $p$ improves this to
\[
|s(\zeta)|\leq C_re^{kM_r/2}\Bigl(\frac{|\zeta|}{r}\Bigr)^{j}\qquad(|\zeta|<r).
\]
For $|\zeta||\tau|<r$, the bounds $q_{k,a}\leq\ord_p(s_{k,a})$ and $(|\zeta|/r)^2|\tau|^2\leq1$ give $|s_{k,a}(\zeta)\tau^{q_{k,a}}|^2\leq C_r^2e^{kM_r}$; summing over the $d_k$ basis elements and taking $\frac1k\log$ proves the bound.
\end{proof}

\begin{lem}\label{lem:schwarz}
Assume $\intr(K)\cap\mathbb Z^n=\{0\}$, fix $p\in X$, let $\psi_{p,t}$ be the envelope of Lemma~\ref{lem:envelope}, and put
\[
Z_p(t)=\frac{1}{\vol(K)}\int_Xe^{-\psi_{p,t}}\,d\nu,\qquad L_p(t)=-\log Z_p(t);
\]
this is the partition function of Lemma~\ref{lem:berndtsson} applied with $w_0=\varphi$ and $V_0=\vol(K)$, whose hypotheses hold by Lemmas~\ref{lem:envelope}, \ref{lem:monomial}, and~\ref{lem:potential}. Then
\[
L_p(t)\leq nt\qquad\text{for every }t\geq0.
\]
\end{lem}

\begin{proof}
Fix $r>0$ and the coordinates of Lemma~\ref{lem:ballbound}. Since $d_k=O(k^n)$ by Lemma~\ref{lem:monomial}, the error term $\frac1k\log(C_r^2d_k)$ tends to zero as $k\to\infty$. The bound of Lemma~\ref{lem:ballbound} holds on the open subset $\{|\zeta|<r|\tau|^{-1}\}$ of $X\times\mathbb C^*$, so it survives the upper-semicontinuous regularization defining the envelope, and we conclude $\psi_{p,t}\leq M_r$ on $\{|\zeta|<re^{-t/2}\}$. The $d\nu$-volume of this coordinate ball is a positive constant times $e^{-nt}$, because its radius in $\mathbb C^n$ is $re^{-t/2}$ and $d\nu$ is a constant multiple of Lebesgue measure there. Hence
\[
Z_p(t)\geq c_re^{-M_r}e^{-nt}\quad\text{for all }t\geq0,\qquad\text{so}\qquad L_p(t)\leq nt+D_r
\]
for a constant $D_r$ independent of $t$. Finally, $L_p$ is convex with $L_p(0)=0$ by Lemma~\ref{lem:berndtsson}, so $t\mapsto L_p(t)/t$ is nondecreasing on $(0,\infty)$; for $0<t<T$ this gives $L_p(t)/t\leq L_p(T)/T\leq n+D_r/T$, and letting $T\to\infty$ removes the constant.
\end{proof}

\subsection{The gateway theorem}\label{subsec:gatewaythm}

\begin{lem}\label{lem:layercake}
For every $k$ and $p$,
\[
\int_Xg_{k,p}\,d\mu_k=\frac{1}{kd_k}\sum_{j=1}^{N_k}\dim F_{k,p}^j=A_k.
\]
\end{lem}

\begin{proof}
By \eqref{eq:slope} and the definition \eqref{eq:bergman} of $\mu_k$, the factor $B_k$ cancels and
\[
\int_Xg_{k,p}\,d\mu_k=\frac{1}{kd_k}\sum_aq_{k,a}\int_X|s_{k,a}|^2e^{-k\varphi}\,d\nu=\frac{1}{kd_k}\sum_aq_{k,a},
\]
using orthonormality. For every $a$ we have $q_{k,a}=\min(j_{k,a},N_k)=\sum_{j=1}^{N_k}\mathbf 1_{j_{k,a}\geq j}$, and $F_{k,p}^j$ is spanned by exactly those $s_{k,a}$ with $j_{k,a}\geq j$; summing first over $a$ and then over $j$ gives $\sum_aq_{k,a}=\sum_{j=1}^{N_k}\dim F_{k,p}^j$. The second equality is \eqref{eq:jetcount}.
\end{proof}

\begin{thm}[Sharp jet-sum theorem]\label{thm:gateway}
Let $K\subset\mathbb R^n$ be a full-dimensional compact convex body with barycenter $0$, $\intr(K)\cap\mathbb Z^n=\{0\}$, and $\vol(K)=(n+1)^n/n!$. Then, with the notation above:
\begin{enumerate}
\item $A_k\to n$ as $k\to\infty$; equivalently, $\displaystyle\frac{1}{kd_k}\sum_{j=1}^{(n+1)k}\bigl(d_k-r_{k,j}\bigr)\longrightarrow n$;
\item for every $p\in X$, the envelope partition function satisfies $L_p(t)=nt$ for every $t\geq0$, and $\int_Xg_p\,d\mu=n$.
\end{enumerate}
\end{thm}

\begin{proof}
Fix $p\in X$. Lemma~\ref{lem:envelope} constructs the envelope $\psi_{p,t}$ with $\psi_{p,0}=\varphi$, slice bounds $0\leq\psi_{p,t}-\varphi\leq(n+1)t$, pointwise initial velocity $g_p\in[0,n+1]$, and $\limsup_kg_{k,p}\leq g_p$ pointwise. The slice convexity in $t$ of the envelope gives, for every $z$, the monotone convergence of difference quotients, so $g_p$ is a pointwise limit of measurable functions, and Lemma~\ref{lem:berndtsson} applies with $w_0=\varphi$ and $c=n+1$ (its one-dimensionality hypothesis holds by Lemma~\ref{lem:monomial} since $\intr(K)\cap\mathbb Z^n=\{0\}$, and $V_0=\vol(K)$ by Lemma~\ref{lem:potential}), so $L_p$ is finite and convex, $L_p(0)=0$, and $L'_{p,+}(0)=\int_Xg_p\,d\mu$. Lemma~\ref{lem:schwarz} gives $L_p(t)\leq nt$.

We now apply Proposition~\ref{prop:bridge} with $g_k=g_{k,p}$, $g=g_p$, $c=n+1$, $A_k=\int g_{k,p}\,d\mu_k$ (Lemma~\ref{lem:layercake}), and the total-variation convergence $\|\mu_k-\mu\|_{\mathrm{TV}}\to0$ of Lemma~\ref{lem:bergmanconv}. It gives $\limsup_kA_k\leq n$.

For the lower bound, Lemma~\ref{lem:jet} gives $\dim F_{k,p}^j\geq d_k-\binom{n+j-1}{n}$, so, summing over $1\leq j\leq N_k$ and using the hockey-stick identity $\sum_{j=1}^{N}\binom{n+j-1}{n}=\binom{n+N}{n+1}$,
\[
A_k\geq\frac{N_k}{k}-\frac{1}{kd_k}\binom{n+N_k}{n+1}.
\]
Here $N_k/k=n+1$, while $d_k/k^n\to\vol(K)=(n+1)^n/n!$ by Lemma~\ref{lem:monomial} and the volume hypothesis, and $\binom{n+N_k}{n+1}/k^{n+1}\to(n+1)^{n+1}/(n+1)!$; hence
\[
\liminf_kA_k\geq(n+1)-\frac{(n+1)^{n+1}/(n+1)!}{(n+1)^n/n!}=(n+1)-1=n.
\]
This is the only step of the present argument that uses the volume hypothesis. The second conclusion of Proposition~\ref{prop:bridge} now yields $A_k\to n$, $L'_{p,+}(0)=\int g_p\,d\mu=n$, and $L_p(t)=nt$ for every $t\geq0$. Assertion (1) and its equivalent discrete form \eqref{eq:jetcount} follow, and (2) holds for every $p$ since $p$ was arbitrary.
\end{proof}

\begin{rem}\label{rem:gatewaysimplex}
By Lemma~\ref{lem:model}, the centered standard simplex $S_n$ satisfies the hypotheses of Theorem~\ref{thm:gateway}, so all conclusions of this section hold for it.
\end{rem}

\section{Hypersurface exclusion and lattice widths}\label{sec:discrete}

In this section, we deduce from Theorem~\ref{thm:gateway} two discrete consequences for a body satisfying the hypotheses of Theorem~\ref{thm:main}: asymptotically, no algebraic hypersurface of degree below $(n+1)k$ passes through all lattice points of $\intr(kK)$, and every lattice width of $K$ is at least $n+1$. The first consequence is used in Sections~\ref{sec:rays} and~\ref{sec:simplex}; the second is of independent interest. On the Fano side, the hypersurface exclusion corresponds to the separation of jets at equality \cite[Proposition~3.2]{Zha25}, and the width bound corresponds to the Seshadri-constant saturation of \cite[Theorem~2.3]{Fuj18}; see Remark~\ref{rem:dictionary}. For a finite set $S\subset\mathbb R^n$, we write $\delta(S)$ for the least total degree of a nonzero real polynomial vanishing at every point of $S$. By Lemma~\ref{lem:jet} and \cite[Proposition~1.1]{Per00}, $\delta(S_k)-1$ is the largest integer $q$ for which $H_k$ generates $q$-jets at a point, and in fact at every point, of the dense torus.

\begin{lem}\label{lem:multiplication}
Let $S\subset\mathbb R^n$ be a finite set with $d$ elements. For $j\geq1$ let $\mathrm{ev}_j$ be the evaluation map from the real polynomials of total degree less than $j$ to functions on $S$, and put $f_j=d-\operatorname{rank}(\mathrm{ev}_j)$. Suppose that a nonzero polynomial $P$ of total degree $D$ vanishes at every point of $S$. Then, for every integer $j>D$,
\[
f_j\geq d-\binom{n+j-1}{n}+\binom{n+j-D-1}{n},
\]
and consequently, for every positive integer $N$,
\[
\sum_{j=1}^Nf_j\geq Nd-\binom{n+N}{n+1}+\binom{n+N-D}{n+1},
\]
where the last binomial coefficient is understood to be zero when $N\leq D$.
\end{lem}

\begin{proof}
Let $q_j$ be the dimension of the kernel of $\mathrm{ev}_j$. Since the domain of $\mathrm{ev}_j$ has dimension $\binom{n+j-1}{n}$, rank--nullity gives
\begin{equation}\label{eq:additive}
f_j=d-\binom{n+j-1}{n}+q_j.
\end{equation}
Fix $j>D$. Multiplication by $P$ is an injective linear map from the polynomials of degree less than $j-D$ into the kernel of $\mathrm{ev}_j$: injectivity holds because the polynomial ring is an integral domain, and the image lies in the kernel because $(PQ)(s)=P(s)Q(s)=0$ for every $s\in S$. Hence $q_j\geq\binom{n+j-D-1}{n}$, proving the first inequality. Summing \eqref{eq:additive} over $1\leq j\leq N$, using $q_j\geq0$ for $j\leq D$ and the bound above for $j>D$, and evaluating the two binomial sums by the hockey-stick identity $\sum_{j=1}^{N}\binom{n+j-1}{n}=\binom{n+N}{n+1}$, proves the second inequality.
\end{proof}

\begin{thm}[Hypersurface exclusion]\label{thm:hypersurface}
Let $K\subset\mathbb R^n$ be a full-dimensional compact convex body with barycenter $0$, $\intr(K)\cap\mathbb Z^n=\{0\}$, and $\vol(K)=(n+1)^n/n!$. Then, for every $\varepsilon>0$ and all sufficiently large $k$, no nonzero real polynomial of total degree at most $(n+1-\varepsilon)k$ vanishes at every point of $S_k$. Equivalently,
\[
\liminf_{k\to\infty}\frac{\delta(S_k)}{k}\geq n+1.
\]
\end{thm}

\begin{proof}
Write $V=(n+1)^n/n!$ and $f_{k,j}=d_k-r_{k,j}$, where $r_{k,j}$ is as in \eqref{eq:jetcount}; note that the evaluation matrix of the monomials of degree less than $j$ on $S_k$ has integer entries, so its rank over $\mathbb R$ equals its rank over $\mathbb C$ and $f_{k,j}$ computes the codimension defect in Lemma~\ref{lem:jet} for real polynomials as well. Theorem~\ref{thm:gateway}(1) says
\begin{equation}\label{eq:jetsumlimit}
\frac{1}{kd_k}\sum_{j=1}^{N_k}f_{k,j}\longrightarrow n,\qquad N_k=(n+1)k.
\end{equation}
We first record the generic part of the sum. Put $g_{k,j}=\max\bigl(d_k-\binom{n+j-1}{n},0\bigr)$, so that $f_{k,j}\geq g_{k,j}$, since the rank $r_{k,j}$ is at most both $d_k$ and $\binom{n+j-1}{n}$. Uniformly for $1\leq j\leq N_k$ we have $\binom{n+j-1}{n}/k^n-(j/k)^n/n!\to0$, and $d_k/k^n\to V$ by Lemma~\ref{lem:monomial}; hence, by Riemann sums,
\begin{equation}\label{eq:genericsum}
\frac{1}{kd_k}\sum_{j=1}^{N_k}g_{k,j}\longrightarrow\frac1V\int_0^{n+1}\Bigl(V-\frac{s^n}{n!}\Bigr)ds=\frac{(n+1)V-V}{V}=n,
\end{equation}
where we used $\int_0^{n+1}s^n\,ds/n!=(n+1)^{n+1}/(n+1)!=V$ and that the integrand is nonnegative on $[0,n+1]$.

Now fix $\varepsilon>0$. It suffices to treat $0<\varepsilon<n+1$: for $\varepsilon=n+1$ a nonzero polynomial of degree at most $0$ is a nonzero constant, which does not vanish at $0\in S_k$, and for $\varepsilon>n+1$ the degree bound is negative and there is no candidate polynomial at all. Suppose, for contradiction, that along an infinite set of $k$ there are nonzero polynomials $P_k$ of degree $D_k\leq(n+1-\varepsilon)k$ vanishing on $S_k$. Let $J_k$ be the set of integers $j$ with $(n+1-\varepsilon/2)k\leq j\leq(n+1-\varepsilon/4)k$. For $j\in J_k$ we have $j/k\leq n+1-\varepsilon/4$, so $d_k-\binom{n+j-1}{n}\geq ak^n$ for a constant $a>0$ depending only on $n$ and $\varepsilon$ and all large $k$; in particular $g_{k,j}=d_k-\binom{n+j-1}{n}$ there. Also $j-D_k\geq\varepsilon k/4$ on $J_k$, so Lemma~\ref{lem:multiplication} gives, for $j\in J_k$,
\[
f_{k,j}\geq g_{k,j}+\binom{n+\lceil\varepsilon k/4\rceil-1}{n}\geq g_{k,j}+bk^n
\]
for a constant $b>0$ depending only on $n$ and $\varepsilon$. The set $J_k$ contains at least $\varepsilon k/8$ integers for all large $k$, while $f_{k,j}\geq g_{k,j}$ elsewhere. Hence, along the supposed subsequence,
\[
\frac{1}{kd_k}\sum_{j=1}^{N_k}f_{k,j}\geq\frac{1}{kd_k}\sum_{j=1}^{N_k}g_{k,j}+\frac{(\varepsilon/8)\,b\,k^{n+1}}{kd_k},
\]
whose lower limit is at least $n+\varepsilon b/(8V)>n$ by \eqref{eq:genericsum}, contradicting \eqref{eq:jetsumlimit}. The equivalent formulation in terms of $\delta(S_k)$ is immediate.
\end{proof}

\begin{lem}\label{lem:principal}
Let $j$ be a positive integer and $b\in\mathbb Z^n$. The principal lattice set
\[
G=\Bigl\{b+a:a\in\mathbb Z_{\geq0}^n,\ \sum_ia_i<j\Bigr\}
\]
has $\binom{n+j-1}{n}$ elements and is unisolvent for real polynomials of total degree less than $j$: the only such polynomial vanishing at every point of $G$ is the zero polynomial.
\end{lem}

\begin{proof}
Translation by $-b$ preserves cardinality, degree, and vanishing, so we may assume $b=0$, in which case $G=\{a\in\mathbb Z_{\geq0}^n:\sum_ia_i\leq j-1\}$, whose cardinality is the number of monomials of degree less than $j$, namely $\binom{n+j-1}{n}$. For a multi-index $\alpha$ with $\sum_i\alpha_i\leq j-1$, let $C_\alpha(x)=\prod_i\binom{x_i}{\alpha_i}$, a polynomial of total degree $\sum_i\alpha_i$ whose leading monomial is a nonzero multiple of $x^\alpha$; the family $(C_\alpha)$ is therefore a basis of the polynomials of degree less than $j$. At a node $\beta\in G$ we have $C_\alpha(\beta)=\prod_i\binom{\beta_i}{\alpha_i}$, which vanishes unless $\alpha_i\leq\beta_i$ for every $i$, and equals $1$ at $\alpha=\beta$. Ordering both the basis and the nodes by a linear extension of the componentwise partial order, the evaluation matrix is triangular with unit diagonal, hence invertible, and unisolvence follows.
\end{proof}

\begin{prop}\label{prop:simplexdelta}
For the centered standard simplex $S_n$ and every positive integer $k$,
\[
\delta\bigl(S_k(S_n)\bigr)=(n+1)k-n.
\]
In particular, the threshold in Theorem~\ref{thm:hypersurface} is asymptotically attained.
\end{prop}

\begin{proof}
By Lemma~\ref{lem:model}, an integer point $m$ lies in $\intr(kS_n)$ exactly when $m_i>-k$ for every $i$ and $\sum_im_i<k$; equivalently, $b=m+(k-1)(1,\dots,1)$ satisfies $b\in\mathbb Z^n_{\geq0}$ and $\sum_ib_i\leq q:=(n+1)(k-1)$. Translation preserves total degree, so it suffices to compute the least vanishing degree of the simplex grid $B_q=\{b\in\mathbb Z_{\geq0}^n:\sum_ib_i\leq q\}$, which is the principal lattice set of Lemma~\ref{lem:principal} with $b=0$ and $j=q+1$. By that lemma, no nonzero polynomial of degree at most $q$ vanishes on $B_q$. On the other hand, $\prod_{r=0}^{q}\bigl(\textstyle\sum_ix_i-r\bigr)$ has degree $q+1$ and vanishes on $B_q$. Therefore $\delta(B_q)=q+1=(n+1)k-n$.
\end{proof}

\begin{cor}\label{cor:width}
Let $K$ be as in Theorem~\ref{thm:hypersurface}. Then $w_u(K)\geq n+1$ for every primitive $u\in\mathbb Z^n$. Moreover, for the centered standard simplex, $w_u(S_n)\geq n+1$ for every primitive $u$, with equality for $u=e_i$ and for $u=(1,\dots,1)$.
\end{cor}

\begin{proof}
Fix a nonzero $u\in\mathbb Z^n$ and let $T_k$ be the set of integers $\langle u,m\rangle$ as $m$ ranges over $S_k$. The polynomial $\prod_{t\in T_k}(\langle u,x\rangle-t)$ is nonzero, has total degree $\#T_k$, and vanishes on $S_k$, so $\delta(S_k)\leq\#T_k$. Every element of $T_k$ lies between $k\min_{x\in K}\langle u,x\rangle$ and $k\max_{x\in K}\langle u,x\rangle$, so $\#T_k\leq kw_u(K)+1$. Combining with Theorem~\ref{thm:hypersurface},
\[
n+1\leq\liminf_k\frac{\delta(S_k)}{k}\leq w_u(K).
\]
For the simplex, we have $w_u(S_n)=(n+1)w_u(\Delta_n)$, and for nonzero integer $u$,
\[
w_u(\Delta_n)=\max(0,u_1,\dots,u_n)-\min(0,u_1,\dots,u_n)\geq1,
\]
with value $1$ at $u=e_i$ and at $u=(1,\dots,1)$.
\end{proof}

\section{Toric limit rays and the pyramid structure}\label{sec:rays}

In this section, we degenerate the point-jet filtrations of Section~\ref{sec:gateway} radially along a generic direction, obtaining monomial filtrations with integer weights; we show that the associated rays subconverge to affine Pr\'ekopa rays, which the equality case of the Pr\'ekopa--Leindler inequality forces to be translation rays; and we prove that a translation ray forces $K$ to be a pyramid with prescribed cap volumes (Theorem~\ref{thm:pyramid}). On the Fano side, the counterpart of the pyramid structure is the Newton--Okounkov cone rigidity of \cite[Proposition~4.6]{Zha22}; the radial degeneration has no Fano-side antecedent, while the Pr\'ekopa--Leindler equality enters as the real counterpart of Berndtsson convexity (Remark~\ref{rem:dictionary}). Throughout, $K$ satisfies the hypotheses of Theorem~\ref{thm:main}, and $\varphi$, $H_k$, $B_k$, $\mu_k$, $\mu$, $S_k$, $d_k$, $N_k$, and the normalized monomials $s_{k,m}$ are as in Sections~\ref{subsec:transport} and~\ref{sec:gateway}.

\subsection{Radial degeneration of nested monomial subspaces}\label{subsec:degeneration}

\begin{lem}\label{lem:grassmann}
Let $S\subset\mathbb Z^n$ be a nonempty finite set and let $H$ be a finite-dimensional complex Hilbert space with orthonormal basis $(e_m)_{m\in S}$ on which the torus $(\mathbb C^*)^n$ acts linearly by $U_ae_m=a^me_m$. Let $F^1\supseteq F^2\supseteq\dots\supseteq F^N$ be a decreasing chain of subspaces of $H$. For a nonzero subspace $F$ of dimension $r$, choose a nonzero $\xi_F\in\bigwedge^rF$, expand $\xi_F=\sum_{|I|=r}c_Ie_I$ in the exterior basis $e_I=\bigwedge_{m\in I}e_m$, formed with respect to a fixed total order on $S$ (the order affects only the signs of the $c_I$, never their vanishing), and let
\[
P(F)=\conv\Bigl\{\sum_{m\in I}m:c_I\neq0\Bigr\}\subset\mathbb R^n
\]
be its weight polytope, which does not depend on the choice of $\xi_F$. Call $v\in\mathbb R^n$ \emph{admissible} for the chain if, for every $j$ with $F^j\neq0$, the linear functional $\langle v,\cdot\rangle$ has a unique maximizer on $P(F^j)$; the inadmissible $v$ lie in a finite union of proper linear hyperplanes, namely the perpendiculars of the nonzero differences of vertices of the finitely many weight polytopes. Fix an admissible $v$ and put $D_t=U_{(e^{tv_1},\dots,e^{tv_n})}$. Then, for each $j$, there is a subset $I_j\subset S$ with $\#I_j=\dim F^j$ such that $D_tF^j$ converges in the Grassmannian to $\operatorname{span}\{e_m:m\in I_j\}$ as $t\to\infty$, and the orthogonal projections $P^j_t$ onto $D_tF^j$ converge in operator norm to the coordinate projections; the limit index sets are nested, $I_1\supseteq I_2\supseteq\dots\supseteq I_N$. Consequently $Q_t=\sum_{j=1}^NP^j_t$ converges in operator norm to the diagonal operator with $Q_\infty e_m=q_me_m$, where $q_m=\#\{j:m\in I_j\}\in\{0,1,\dots,N\}$ and
\[
\sum_{m\in S}q_m=\sum_{j=1}^N\dim F^j.
\]
Moreover, if $e_m(z)=c_mz^m$ with $c_m\neq0$ and $F^j=F_{1}^j$ is the subspace of functions vanishing to order at least $j$ at $1=(1,\dots,1)$, then for $p_t=(e^{-tv_1},\dots,e^{-tv_n})$ one has $F_{p_t}^j=D_tF_1^j$, so the orthogonal projection onto $F_{p_t}^j$ is $P_t^j$.
\end{lem}

\begin{proof}
The polytope $P(F)$ is well defined since $\xi_F$ is unique up to a nonzero scalar. For the hyperplane description of inadmissibility: if $\langle v,\cdot\rangle$ has a non-unique maximizer on a polytope, then the maximizing face has positive dimension and contains two distinct vertices of the polytope, on which $\langle v,\cdot\rangle$ takes equal values; this confines $v$ to the union of the perpendicular hyperplanes of the nonzero vertex differences, and over the finitely many $j$ that union is a finite union of proper hyperplanes.

Fix $j$ with $r=\dim F^j>0$ and $\xi_j=\sum_Ic_Ie_I$ as above. Since $D_te_m=e^{t\langle v,m\rangle}e_m$,
\[
{\textstyle\bigwedge^r}D_t(\xi_j)=\sum_Ic_Ie^{t\langle v,\sum_{m\in I}m\rangle}e_I.
\]
Let $a_j$ be the unique maximizer of $\langle v,\cdot\rangle$ on $P(F^j)$; it is a vertex of the weight polytope, hence one of the lattice points $\sum_{m\in I}m\in\mathbb Z^n$. Dividing by $e^{t\langle v,a_j\rangle}$ and letting $t\to\infty$, every term whose weight sum differs from $a_j$ tends to zero, so the projective class of $\bigwedge^rD_t(\xi_j)$ converges to that of the nonzero vector $\xi_{j,\infty}=\sum_{\sum_{m\in I}m=a_j}c_Ie_I$. The Pl\"ucker embedding has closed image, so $\xi_{j,\infty}$ is decomposable and represents the Grassmannian limit $F_{j,\infty}$ of $D_tF^j$. Each $e_I$ occurring in $\xi_{j,\infty}$ satisfies $U_ae_I=a^{a_j}e_I$, so the Pl\"ucker line of $F_{j,\infty}$ is fixed by every $\bigwedge^rU_a$, and $F_{j,\infty}$ is invariant under every $U_a$ by the equivariance and injectivity of the Pl\"ucker embedding. Choose $b\in(\mathbb C^*)^n$ such that the numbers $b^m$, $m\in S$, are pairwise distinct; the operator $U_b$ is then diagonal with simple spectrum, its spectral projections are polynomials in $U_b$, and they preserve the invariant subspace $F_{j,\infty}$, which is therefore spanned by a subset $\{e_m:m\in I_j\}$ of the basis, with $\#I_j=r$. Operator-norm convergence of the orthogonal projections follows from continuity of the projection on the finite-dimensional Grassmannian. The incidence relation $A\subseteq B$ is closed in the product of Grassmannians, so nestedness passes to the limit, and coordinate subspaces are nested exactly when their index sets are. Summing the finitely many operator-norm limits gives the diagonal limit of $Q_t$, and taking traces gives $\sum_mq_m=\sum_j\dim F^j$.

For the point-jet application, define $(U_af)(z)=f(az)$; a function vanishes to order at least $j$ at $p$ if and only if $U_af$ vanishes to order at least $j$ at $a^{-1}p$, so $F_{a^{-1}p}^j=U_aF_p^j$. Taking $p=1$ and $a=(e^{tv_1},\dots,e^{tv_n})$ gives $F_{p_t}^j=D_tF_1^j$.
\end{proof}

For $p\in X$ and $1\leq j\leq N_k$, we write $P_{k,p}^j$ for the orthogonal projection of $H_k$ onto $F_{k,p}^j$, and we put $Q_{k,p}=\sum_{j=1}^{N_k}P_{k,p}^j$.

\begin{prop}\label{prop:weights}
There is a countable union $E$ of proper linear hyperplanes of $\mathbb R^n$, depending only on $K$ and $\varphi$, with the following property. For every $v\in\mathbb R^n\setminus E$ and every positive integer $k$ simultaneously, the direction $v$ is admissible for the chain $(F_{k,1}^j)_{1\leq j\leq N_k}$ in the sense of Lemma~\ref{lem:grassmann}, and hence, with $p_R=(e^{-Rv_1},\dots,e^{-Rv_n})$:
\begin{enumerate}
\item $Q_{k,p_R}=\sum_{j=1}^{N_k}P_{k,p_R}^j$ converges, as $R\to\infty$, in operator norm to a diagonal operator with $Q_{k,\infty}s_{k,m}=q_{k,m}s_{k,m}$, where the weights $q_{k,m}$ are integers in $[0,N_k]$ satisfying
\[
\sum_{m\in S_k}q_{k,m}=\sum_{j=1}^{N_k}\dim F_{k,1}^j,\qquad\text{and hence}\qquad\frac{1}{kd_k}\sum_{m\in S_k}q_{k,m}\longrightarrow n\quad(k\to\infty);
\]
moreover, for $1\leq j\leq N_k$, the limit index set of $D_RF_{k,1}^j$ supplied by Lemma~\ref{lem:grassmann}, where $D_R=U_{(e^{Rv_1},\dots,e^{Rv_n})}$ in the notation of that lemma, is $I_{k,j}=\{m\in S_k:q_{k,m}\geq j\}$, these sets are nested in $j$, and $\#\{m\in S_k:q_{k,m}<j\}=\operatorname{codim}F_{k,1}^j$;
\item for every fixed $k$, the finite rays at $p_R$ defined by \eqref{eq:finiteray} converge, as $R\to\infty$, locally uniformly in $(z,t)$ to the toric monomial ray
\[
u_{k,t}(z)=\frac1k\log\sum_{m\in S_k}|s_{k,m}(z)|^2e^{tq_{k,m}}.
\]
\end{enumerate}
\end{prop}

\begin{proof}
For each $k$, Lemma~\ref{lem:grassmann} confines the inadmissible directions for the chain $(F_{k,1}^j)_j$ to a finite union of proper hyperplanes; let $E$ be the union over all $k$, a countable union of proper hyperplanes depending only on the chains, hence only on $K$ and $\varphi$. Fix $v\notin E$. Assertion (1) is Lemma~\ref{lem:grassmann} applied at each $k$ (with $H=H_k$, $S=S_k$, $e_m=s_{k,m}$, $N=N_k$, and $F^j=F_{k,1}^j$), followed by \eqref{eq:jetcount} and Theorem~\ref{thm:gateway}(1) for the normalized trace limit; the identification $I_{k,j}=\{m:q_{k,m}\geq j\}$ and the codimension count follow from the nestedness of the index sets together with $q_{k,m}=\#\{j:m\in I_{k,j}\}$ and $\#I_{k,j}=\dim F_{k,1}^j$. For (2), the projections $P_{k,p}^j$ onto the nested subspaces commute, and an orthonormal basis adapted to the filtration diagonalizes them all; on a basis vector of exact vanishing order $d$ at $p$, the operator $Q_{k,p}$ acts by $\min(d,N_k)$. If $\kappa_z\in H_k$ denotes the evaluation vector at $z$, characterized by $f(z)=\langle f,\kappa_z\rangle$, then
\[
\sum_a|s_{k,a}(z)|^2e^{t\min(\ord_p(s_{k,a}),N_k)}=\bigl\langle e^{tQ_{k,p}}\kappa_z,\kappa_z\bigr\rangle.
\]
Operator-norm convergence $Q_{k,p_R}\to Q_{k,\infty}$ implies operator-norm convergence of $e^{tQ_{k,p_R}}$, uniformly for $t$ in compact intervals; since $Q_{k,\infty}$ is diagonal in the monomial basis, the right side converges to $\sum_m|s_{k,m}(z)|^2e^{tq_{k,m}}$. The evaluation vector depends continuously on $z$ with locally bounded norm, so the displayed inner products converge to $\sum_m|s_{k,m}(z)|^2e^{tq_{k,m}}$ locally uniformly in $(z,t)$. Since $Q_{k,p}$ is positive semidefinite, all the quantities involved dominate $\langle\kappa_z,\kappa_z\rangle=B_k(z)$, and $B_k$ is continuous and positive, hence bounded below by a positive constant on every compact subset of $X$; as the logarithm is Lipschitz on $[\varepsilon,\infty)$ for every $\varepsilon>0$, taking $\frac1k\log$ preserves local uniform convergence.
\end{proof}

\subsection{Toric limit rays}\label{subsec:limitrays}

\begin{lem}\label{lem:finitelevel}
Let $v\in\mathbb R^n\setminus E$ with $E$ as in Proposition~\ref{prop:weights}, and let $k$ be a positive integer such that $P_k=\frac1k\conv(S_k)$ is full-dimensional with $0\in\intr(P_k)$. Then, for the associated toric monomial ray $u_{k,t}$ of Proposition~\ref{prop:weights} and $t\geq0$, the integrals $\int_Xe^{-u_{k,t}}\,d\nu$ are finite and positive, and
\[
L_k(t)=-\log\frac{\int_Xe^{-u_{k,t}}\,d\nu}{\int_Xe^{-u_{k,0}}\,d\nu}
\]
is convex and satisfies $L_k(t)\leq nt$.
\end{lem}

\begin{proof}
The initial weight $u_{k,0}=\frac1k\log B_k$ differs by a bounded function from $h_{P_k}$: indeed $B_k=\sum_me^{\langle m,x\rangle}/I_k(m/k)$ is a finite positive combination of the $e^{\langle m,x\rangle}$ with $m/k\in P_k$, and $\max_{m\in S_k}\langle m/k,x\rangle=h_{P_k}(x)$, so $\frac1k\log B_k-h_{P_k}$ is bounded by $\frac1k$ times the logarithms of $d_k$ and of the finitely many normalizing constants. Moreover $P_k\subseteq K$, so $\intr(P_k)\cap\mathbb Z^n\subseteq\intr(K)\cap\mathbb Z^n=\{0\}$, and Lemma~\ref{lem:basis}, applied with the body $P_k$ and the weight $u_{k,0}$, shows that the weighted holomorphic space of the weight $u_{k,0}$ consists exactly of the constants and that $\int_Xe^{-u_{k,0}}\,d\nu\in(0,\infty)$; the same then holds for every ray of the form
\[
w_t=\frac1k\log\sum_a|s_{k,a}|^2e^{tq_a},
\]
where $(s_{k,a})_a$ is any orthonormal basis of $H_k$ and $0\leq q_a\leq N_k$ are integers, since $w_0=u_{k,0}$ (every orthonormal basis satisfies $\sum_a|s_{k,a}|^2=B_k$) and $0\leq w_t-u_{k,0}\leq(n+1)t$ is bounded for fixed $t$; this covers both the toric ray $u_{k,t}$, obtained from the monomial basis with the weights $q_{k,m}$, and the finite rays \eqref{eq:finiteray} based at the points $p_R$. In particular all the integrals $\int_Xe^{-w_t}\,d\nu$ are finite and positive. Each total-space function $\frac1k\log\sum_a|s_{k,a}(z)\tau^{q_a}|^2$ is plurisubharmonic on $X\times\mathbb C^*$ and rotation-invariant in $\tau$, with radial slices $w_t$; Lemma~\ref{lem:berndtsson}, applied with $w_0=u_{k,0}$, $V_0=\int_Xe^{-u_{k,0}}\,d\nu$, and $c=n+1$, shows that every relative partition function
\[
L^w(t)=-\log\frac{\int_Xe^{-w_t}\,d\nu}{\int_Xe^{-u_{k,0}}\,d\nu}
\]
is finite and convex on $[0,\infty)$ with $L^w(0)=0$; in particular so is $L_k$.

For the upper bound, we use the finite rays at the radially escaping points. Fix $R$ and let $w_t$ be the ray \eqref{eq:finiteray} based at $p_R$. Lemma~\ref{lem:ballbound}, applied at the base point $p_R$ with a coordinate radius $r$, bounds $w_t$ by $M+\tfrac1k\log(C^2d_k)$ on the coordinate ball of radius $re^{-t/2}$ about $p_R$, whose $d\nu$-volume is a positive constant times $e^{-nt}$; here $M$ and $C$ depend on $p_R$ and $r$ but not on $t$. Hence $\int_Xe^{-w_t}\,d\nu\geq c'e^{-nt}$ with $c'$ independent of $t$, so $L^w(t)\leq nt+D$ for a constant $D$ independent of $t$. Since $L^w$ is convex with $L^w(0)=0$, the quotient $L^w(t)/t$ is nondecreasing, so for $0<t<T$ we get $L^w(t)/t\leq L^w(T)/T\leq n+D/T$, and $T\to\infty$ gives $L^w(t)\leq nt$ for every finite ray. By Proposition~\ref{prop:weights}(2), $u_{k,t}$ is the locally uniform limit, as $R\to\infty$, of these rays; since all rays dominate $u_{k,0}$ and $e^{-u_{k,0}}$ is integrable, dominated convergence passes the exponential integrals to the limit, and the bound $L_k(t)\leq nt$ survives.
\end{proof}

\begin{thm}\label{thm:limitray}
Let $v\in\mathbb R^n\setminus E$ with $E$ as in Proposition~\ref{prop:weights}. Every sequence of positive integers tending to infinity has a subsequence along which the toric monomial rays $u_{k,t}$ converge locally uniformly on $X\times[0,\infty)$ to a torus-invariant ray $\psi_t$ with the following properties: $\psi_0=\varphi$; $0\leq\psi_t-\varphi\leq(n+1)t$; the function $(x,t)\mapsto\psi_t(x)$ is convex on $\mathbb R^n\times[0,\infty)$; and the normalized partition function
\[
L(t)=-\log\Bigl(\frac{1}{\vol(K)}\int_Xe^{-\psi_t}\,d\nu\Bigr)
\]
satisfies $L(t)=nt$ for every $t\geq0$.
\end{thm}

\begin{proof}
In the logarithmic coordinates $x$, each $u_{k,t}$ is torus-invariant and convex in $(x,t)$: it is $\frac1k$ times a log-sum-exp of the affine functions $\langle m,x\rangle+tq_{k,m}$ plus constants. Its $x$-gradients lie in $\conv(S_k/k)\subseteq K$ and its $t$-derivative lies in $[0,n+1]$, so the family is equi-Lipschitz on compact subsets of $\mathbb R^n\times[0,\infty)$, and $u_{k,0}(0)\to\varphi(0)$ by Lemma~\ref{lem:bergmanconv}. By the Arzel\`a--Ascoli theorem and a diagonal extraction, every sequence of $k$'s has a subsequence along which $u_{k,t}$ converges locally uniformly to a finite convex function $\psi_t(x)$, jointly in $(x,t)$; local uniform limits of torus-invariant plurisubharmonic functions are torus-invariant plurisubharmonic, and passing to the limit in $u_{k,0}\leq u_{k,t}\leq u_{k,0}+(n+1)t$ together with $u_{k,0}\to\varphi$ gives $\psi_0=\varphi$ and the two-sided bounds.

We next check convergence of the exponential integrals along the subsequence. Since $0\in\intr(K)$ and $\varphi-h_K$ is bounded, there are $\rho>0$ and $C$ with $\varphi(x)\geq\rho|x|-C$. Choose $R_0$ with $\varphi(R_0\theta)-\varphi(0)\geq\rho R_0/2$ for every unit vector $\theta$. Local uniform convergence $u_{k,0}\to\varphi$ gives, for all large $k$ in the subsequence, $u_{k,0}(R_0\theta)-u_{k,0}(0)\geq\rho R_0/3$ uniformly in $\theta$; for $r\geq R_0$, convexity makes the secant slopes of $r\mapsto u_{k,0}(r\theta)$ nondecreasing, so $u_{k,0}(r\theta)\geq u_{k,0}(0)+\frac{\rho}{3}(r-R_0)$. After enlarging a constant to cover $|x|\leq R_0$, there are $\delta>0$, $C_1$, and $k_0$ with
\[
u_{k,t}(x)\geq u_{k,0}(x)\geq\delta|x|-C_1\qquad(k\geq k_0\text{ in the subsequence}),
\]
an integrable exponential bound independent of $k$ and of $t$ in compact intervals. Fix a compact interval of $t$'s. For $\sigma>0$, split the integrals over $\{|x|\leq\sigma\}$ and $\{|x|>\sigma\}$: the tail contributions are at most $\int_{|x|>\sigma}e^{-\delta|x|+C_1}\,dx$, which tends to $0$ as $\sigma\to\infty$ uniformly in $k$ and $t$, while on the compact set $\{|x|\leq\sigma\}$ the uniform convergence $u_{k,t}\to\psi_t$ makes the integrals converge uniformly in $t$. Letting $\sigma\to\infty$ gives $\int_Xe^{-u_{k,t}}\,d\nu\to\int_Xe^{-\psi_t}\,d\nu$, uniformly for $t$ in the interval; the limits are finite and positive.

For all large $k$, the scaled polytope $P_k=\frac1k\conv(S_k)$ is full-dimensional with $0\in\intr(P_k)$: for each $i$, both $\pm e_i\in\intr(kK)$ once $k$ is large (as $0\in\intr(K)$), so $\pm e_i/k\in P_k$ and $P_k$ contains a cross-polytope neighborhood of $0$. Lemma~\ref{lem:finitelevel} therefore applies and gives, for every such $k$,
\[
-\log\frac{\int_Xe^{-u_{k,t}}\,d\nu}{\int_Xe^{-u_{k,0}}\,d\nu}\leq nt,
\]
with the left side convex in $t$. Passing to the limit along the subsequence using the convergence of integrals proved above and $\int_Xe^{-\varphi}\,d\nu=\vol(K)$ (Lemma~\ref{lem:potential}) yields that $L$ is convex with $L(t)\leq nt$.

For the lower bound on the initial slope, define the finite initial slopes
\[
g_k(z)=\frac{\sum_{m\in S_k}(q_{k,m}/k)|s_{k,m}(z)|^2}{B_k(z)}\in[0,n+1].
\]
Orthonormality gives
\[
\int_Xg_k\,d\mu_k=\frac{1}{kd_k}\sum_{m\in S_k}q_{k,m}\longrightarrow n
\]
by Proposition~\ref{prop:weights}(1). For fixed $z$, the functions $t\mapsto u_{k,t}(z)$ are convex with initial slopes $g_k(z)$, $u_{k,0}(z)\to\varphi(z)$, and $u_{k,t}(z)\to\psi_t(z)$; Lemma~\ref{lem:secant} gives $\limsup_kg_k(z)\leq g(z):=\lim_{t\downarrow0}(\psi_t(z)-\varphi(z))/t$, which exists by convexity and lies in $[0,n+1]$. Lemma~\ref{lem:fatou} and Lemma~\ref{lem:derivative} (applied with the measure $\mu$ and $c=n+1$) give
\[
n=\lim_k\int_Xg_k\,d\mu_k\leq\int_Xg\,d\mu=L'_+(0).
\]
Convexity, $L(0)=0$, $L'_+(0)\geq n$, and $L(t)\leq nt$ now pinch $L(t)=nt$ for all $t\geq0$ by Lemma~\ref{lem:pinch}.
\end{proof}

\subsection{Pr\'ekopa equality: affine rays are translation rays}\label{subsec:prekopa}

\begin{lem}\label{lem:affineconsistency}
Let $V\colon\mathbb R^n\times[0,\infty)\to\mathbb R$ be a finite convex function such that $Z(t)=\int_{\mathbb R^n}e^{-V(x,t)}\,dx$ is finite and positive for every $t\geq0$ and $-\log Z(t)=-\log Z(0)+bt$ for a constant $b$. Assume that for every $t\geq0$ there exists $c_t\in\mathbb R^n$ with
\[
\frac{e^{-V(x+c_t,t)}}{Z(t)}=\frac{e^{-V(x,0)}}{Z(0)}\qquad\text{for almost every }x.
\]
Then there is a single vector $c\in\mathbb R^n$ with $c_t=tc$ and $V(x,t)=V(x-tc,0)+bt$ for every $x$ and $t$.
\end{lem}

\begin{proof}
Taking logarithms and using the affine formula for $-\log Z$ gives $V(x+c_t,t)=V(x,0)+bt$ almost everywhere; both sides are continuous in $x$, so the identity holds everywhere, and with $W=V(\cdot,0)$ we obtain $V(x,t)=W(x-c_t)+bt$. The vector $c_t$ is unique: if $c_t$ and $d_t$ both work, then $W$ is invariant under translation by $h=c_t-d_t$; a finite convex function invariant under a nonzero translation is constant on lines parallel to $h$, making $e^{-W}$ non-integrable by Fubini, a contradiction; hence $h=0$, and in particular $c_0=0$.

Fix $0\leq s<t$ and $\lambda\in[0,1]$ and put $r=(1-\lambda)s+\lambda t$. Convexity of $V$ applied to the points $(x+c_s,s)$ and $(x+c_t,t)$ gives
\[
V\bigl(x+(1-\lambda)c_s+\lambda c_t,\,r\bigr)\leq(1-\lambda)V(x+c_s,s)+\lambda V(x+c_t,t)=W(x)+br,
\]
while the slice representation makes the left side $W\bigl(x+(1-\lambda)c_s+\lambda c_t-c_r\bigr)+br$. Hence $W(x+d)\leq W(x)$ for every $x$, where $d=(1-\lambda)c_s+\lambda c_t-c_r$; iterating, $W(x+kd)\leq W(x)$ for every positive integer $k$. If $d\neq0$, coercivity of $W$ would be contradicted. Coercivity holds because $e^{-W}$ is integrable and $W$ is finite convex: if $W\leq M$ on a sequence $x_j$ with $|x_j|\to\infty$ and, after passing to a subsequence, $x_j/|x_j|\to\theta$, then $W\leq\max(W(0),M)$ on the segments $[0,x_j]$, hence on the ray $\mathbb R_{\geq0}\theta$ by continuity of the finite convex function $W$; taking a ball $B(0,\rho)$ on which $W\leq M'$, convexity bounds $W$ by $\max(M',W(0),M)$ on $\conv\bigl(B(0,\rho)\cup\mathbb R_{\geq0}\theta\bigr)$, which contains a half-cylinder of radius $\rho/2$ and infinite volume, making $e^{-W}$ non-integrable. Thus $c_r=(1-\lambda)c_s+\lambda c_t$: the map $t\mapsto c_t$ is affine with $c_0=0$, so $c_t=tc_1$ for all $t\geq0$.
\end{proof}

\begin{prop}\label{prop:prekopa}
Let $V\colon\mathbb R^n\times[0,\infty)\to\mathbb R$ be a finite convex function such that $Z(t)=\int_{\mathbb R^n}e^{-V(x,t)}\,dx$ is finite and positive for every $t\geq0$ and $-\log Z(t)=-\log Z(0)+bt$ for a constant $b$. Then there is a vector $c\in\mathbb R^n$ such that
\[
V(x,t)=V(x-tc,0)+bt\qquad\text{for every }x\in\mathbb R^n,\ t\geq0.
\]
\end{prop}

\begin{proof}
Fix $t>0$ and put $f_0=e^{-V(\cdot,0)}$, $f_t=e^{-V(\cdot,t)}$, and $h=e^{-V(\cdot,t/2)}$. Convexity of $V$ gives
\[
V\Bigl(\frac{x+y}{2},\frac t2\Bigr)\leq\frac12V(x,0)+\frac12V(y,t),\qquad\text{i.e.}\qquad h\Bigl(\frac{x+y}{2}\Bigr)\geq f_0(x)^{1/2}f_t(y)^{1/2},
\]
so the triple $(f_0,f_t,h)$ satisfies the hypothesis of the Pr\'ekopa--Leindler inequality at the fixed weights $(\tfrac12,\tfrac12)$, and the affine formula for $-\log Z$ gives $\int h=Z(t/2)=Z(0)^{1/2}Z(t)^{1/2}=(\int f_0)^{1/2}(\int f_t)^{1/2}$: equality holds. We invoke the equality characterization of the Pr\'ekopa--Leindler inequality for two functions at fixed weights, due to Dubuc \cite[Th\'eor\`eme~12]{Dub77} (see also \cite[Theorem~4.1]{KW22} for the many-function version, whose proof reduces to Dubuc's theorem): if equality holds, there are translation vectors and a log-concave probability density $e^{-\psi}$ such that, after modification on null sets, the normalized densities $f_0/\!\int f_0$ and $f_t/\!\int f_t$ are translates of $e^{-\psi}$; in particular there is $c_t\in\mathbb R^n$ with $e^{-V(x+c_t,t)}/Z(t)=e^{-V(x,0)}/Z(0)$ for almost every $x$. At $t=0$ this holds trivially with $c_0=0$. Lemma~\ref{lem:affineconsistency} now produces a single $c$ with $c_t=tc$ and the asserted identity.
\end{proof}

\subsection{The pyramid theorem}\label{subsec:pyramid}

\begin{thm}[Pyramid structure]\label{thm:pyramid}
Let $v\in\mathbb R^n\setminus E$, and let $\psi_t$ be a subsequential toric limit ray as in Theorem~\ref{thm:limitray}, taken along a subsequence of $k$'s that we fix. Assume that $\psi_t(x)=\varphi(x-tc)+nt$ for a vector $c\in\mathbb R^n$ and all $x,t$. Then $c\neq0$. Put
\[
\ell(y)=n-\langle c,y\rangle\qquad(y\in K).
\]
Then $0\leq\ell\leq n+1$ on $K$, and for every $s\in[0,n+1]$ the cap $C_s=\{y\in K:\ell(y)\leq s\}$ satisfies
\[
\vol(C_s)=\frac{s^n}{n!}.
\]
Moreover $C_0$ consists of a single point $p$, one has
\[
C_s=p+\frac{s}{n+1}(K-p)\qquad(0\leq s\leq n+1),
\]
and, with $B=\{y\in K:\ell(y)=n+1\}$, the body $K$ is the pyramid $K=\conv(\{p\}\cup B)$ with apex $p$ and base $B$; finally $h_K(c)=n$ and $h_K(-c)=1$.
\end{thm}

\begin{proof}
Along the fixed subsequence, define finite measures on $K\times[0,n+1]$ by
\[
\Gamma_k=\frac{1}{k^n}\sum_{m\in S_k}\delta_{(m/k,\,q_{k,m}/k)},
\]
where $\delta_{(y,r)}$ denotes the Dirac mass at the point of $K\times[0,n+1]$ with coordinates $(y,r)$.
Their masses $d_k/k^n$ tend to $\vol(K)$, and their first marginals converge weakly to Lebesgue measure on $K$, both by Lemma~\ref{lem:monomial} and lattice Riemann sums; here the boundary of the full-dimensional compact convex body $K$ is Lebesgue-negligible, since comparing $K$ with its homotheties $x_0+\lambda(K-x_0)$ from an interior point $x_0$ and letting $\lambda\uparrow1$ shows $\vol(\partial K)=0$. Proposition~\ref{prop:weights}(1) gives
\begin{equation}\label{eq:gammamass}
\int r\,d\Gamma_k=\frac{1}{k^n}\sum_{m\in S_k}\frac{q_{k,m}}{k}\longrightarrow n\vol(K).
\end{equation}
Pass to a further subsequence along which $\Gamma_k$ converges weakly to a finite measure $\Gamma$; its first marginal is Lebesgue measure on $K$ and $\int r\,d\Gamma=n\vol(K)$.

\smallskip\noindent\textbf{Step 1.} In this step, we show $r\leq\ell(y)$ for $\Gamma$-almost every $(y,r)$. The first marginal gives no mass to $\partial K\times[0,n+1]$, so it suffices to consider $(y,r)\in\operatorname{supp}\Gamma$ with $y\in\intr(K)$. Choose atoms $(m_k/k,q_{k,m_k}/k)\to(y,r)$: such a choice is possible because every ball around a point of $\operatorname{supp}\Gamma$ has positive $\Gamma$-mass, hence, by weak convergence and the Portmanteau theorem, contains atoms of $\Gamma_k$ for all large $k$ in the subsequence, and a diagonal choice over shrinking balls produces the sequence. By Proposition~\ref{prop:laplace}, for every $x\in\mathbb R^n$,
\[
\frac1k\log|s_{k,m_k}(x)|^2\longrightarrow\langle y,x\rangle-\varphi^*(y).
\]
Since the term indexed by $m_k$ is one summand in the definition of $u_{k,t}$, local uniform convergence $u_{k,t}\to\psi_t$ gives
\[
\psi_t(x)\geq\langle y,x\rangle-\varphi^*(y)+tr\qquad(x\in\mathbb R^n,\ t\geq0).
\]
Taking Legendre transforms in $x$ yields $\psi_t^*(y)\leq\varphi^*(y)-tr$. On the other hand, the translation formula gives directly
\[
\psi_t^*(y)=\sup_x\bigl(\langle y,x\rangle-\varphi(x-tc)-nt\bigr)=\varphi^*(y)+t\langle c,y\rangle-nt=\varphi^*(y)-t\,\ell(y).
\]
Comparing at any $t>0$ proves $r\leq\ell(y)$.

\smallskip\noindent\textbf{Step 2.} In this step, we upgrade the inequality to the identity $r=\ell(y)$ $\Gamma$-almost everywhere. Since the first marginal of $\Gamma$ is Lebesgue measure on $K$ and the barycenter of $K$ is $0$,
\[
\int_K\ell(y)\,dy=n\vol(K)-\Bigl\langle c,\int_Ky\,dy\Bigr\rangle=n\vol(K)=\int r\,d\Gamma,
\]
and $r\leq\ell(y)$ almost everywhere forces $r=\ell(y)$ $\Gamma$-almost everywhere.

\smallskip\noindent\textbf{Step 3.} In this step, we compute the sublevel counts and prove $c\neq0$ and the cap-volume formula. Proposition~\ref{prop:weights}(1) identifies the nested limit index sets and gives
\begin{equation}\label{eq:levelcount}
\#\{m\in S_k:q_{k,m}<j\}=\operatorname{codim}F_{k,1}^j\qquad(1\leq j\leq N_k).
\end{equation}
Fix $0<s<n+1$ and integers $j_k$ with $j_k/k\to s$, and choose $\varepsilon>0$ with $s<n+1-\varepsilon$, so that $j_k-1\leq(n+1-\varepsilon)k$ for all large $k$. By Theorem~\ref{thm:hypersurface}, for all large $k$ no nonzero polynomial of degree less than $j_k$ vanishes on $S_k$, so Lemma~\ref{lem:jet} gives $\operatorname{codim}F_{k,1}^{j_k}=\binom{n+j_k-1}{n}$, and therefore
\begin{equation}\label{eq:capcount}
\frac{1}{k^n}\#\{m\in S_k:q_{k,m}<j_k\}\longrightarrow\frac{s^n}{n!}.
\end{equation}
We record how the moving threshold interacts with weak convergence: the count in \eqref{eq:capcount} is $k^n\,\Gamma_k\bigl(K\times[0,j_k/k)\bigr)$, and for every $\eta>0$ we eventually have $s-\eta<j_k/k<s+\eta$, so the Portmanteau theorem sandwiches
\[
\Gamma\bigl(K\times[0,s-\eta)\bigr)\leq\liminf_k\Gamma_k\bigl(K\times[0,j_k/k)\bigr)\leq\limsup_k\Gamma_k\bigl(K\times[0,j_k/k)\bigr)\leq\Gamma\bigl(K\times[0,s+\eta]\bigr).
\]
If $c$ were zero, then $\ell\equiv n$ and Step~2 would put $\Gamma$ on the level $\{r=n\}$; taking $s\in(0,n)$ and $\eta$ with $s+\eta<n$, the sandwich would force the positive limit $s^n/n!$ of \eqref{eq:capcount} to be at most $\Gamma(K\times[0,s+\eta])=0$, a contradiction. Hence $c\neq0$, and every level set of the affine function $\ell$ is Lebesgue-negligible. By Step~2 and the identification of the first marginal, $\Gamma(K\times[0,u))=\vol(\{y\in K:\ell(y)<u\})$ and $\Gamma(K\times[0,u])=\vol(\{y\in K:\ell(y)\leq u\})$ for every $u$; the sandwich and \eqref{eq:capcount} therefore give $\vol(\{\ell<s-\eta\})\leq s^n/n!\leq\vol(\{\ell\leq s+\eta\})$ for every $\eta>0$, and letting $\eta\downarrow0$, using the negligibility of the level sets of $\ell$, yields, for $0<s<n+1$,
\[
\vol(\{y\in K:\ell(y)<s\})=\frac{s^n}{n!}.
\]
The bounds $0\leq q_{k,m}/k\leq n+1$ and $r=\ell(y)$ give $0\leq\ell\leq n+1$ almost everywhere on $K$, hence everywhere: if $\ell$ violated one of the bounds at a point $y_0\in K$, then, by continuity, it would violate it on the intersection of a ball around $y_0$ with $\intr(K)$, a nonempty open set of positive Lebesgue measure. The strict and non-strict sublevel sets have equal volume, and continuity in $s$ extends the formula to $s\in[0,n+1]$. The formula forces $\min_K\ell=0$ and $\max_K\ell=n+1$: a positive minimum would contradict the formula for small $s$, and a maximum $M<n+1$ would give $\vol(K)=\vol(C_M)=M^n/n!<(n+1)^n/n!$, contradicting the volume hypothesis. Precisely, $\min_K\ell=n-h_K(c)=0$ and $\max_K\ell=n+h_K(-c)=n+1$ give $h_K(c)=n$ and $h_K(-c)=1$.

\smallskip\noindent\textbf{Step 4.} We conclude the proof in this step. Put $T=n+1$ and choose $p\in C_0$, which is nonempty and compact. Convexity gives, for $0<s<T$,
\[
p+\frac sT(K-p)\subseteq C_s,
\]
since $\ell$ is affine with $\ell(p)=0$ and $\ell\leq T$ on $K$. The left side has volume $(s/T)^n\vol(K)=s^n/n!$, which equals $\vol(C_s)$; two nested compact convex bodies of equal volume coincide, so
\begin{equation}\label{eq:homothety}
C_s=p+\frac sT(K-p)\qquad(0<s<T),
\end{equation}
and the identity extends to the endpoints: at $s=T$ both sides equal $K$, while at $s=0$ we have $C_0=\bigcap_{0<s<T}C_s=\bigcap_{0<s<T}\bigl(p+\tfrac sT(K-p)\bigr)=\{p\}$ by the boundedness of $K$, so $C_0=\{p\}$ and \eqref{eq:homothety} holds for all $s\in[0,T]$. For $y\in K$ with $a=\ell(y)>0$, \eqref{eq:homothety} at $s=a$ writes $y=p+\frac aT(z-p)$ with $z\in K$; applying $\ell$ gives $\ell(z)=T$, so $z\in B$ and $y\in\conv(\{p\}\cup B)$. The reverse inclusion is convexity, so $K=\conv(\{p\}\cup B)$.
\end{proof}

\section{Equality bodies are simplices}\label{sec:simplex}

In this section, we complete the analytic half: the apex of the pyramid of Theorem~\ref{thm:pyramid} minimizes the chosen radial direction (Theorem~\ref{thm:apex}); since the good directions are dense, a convexity argument forces $K$ to be a simplex (Theorem~\ref{thm:simplex}); and the volume and barycenter hypotheses then reduce Theorem~\ref{thm:main} to a critical-lattice statement for $S_n$ (Proposition~\ref{prop:reduction}). On the Fano side, this step is carried out by characterizations of $\mathbb P^n$ (\cite[Theorem~1.1]{Fuj18}; \cite[Theorems~13(2) and~36]{Liu18}; \cite[Theorems~2 and~10]{LZ18}; \cite[Theorem~3.3]{Zha25}; \cite{LM25}), none of which has a convex-geometric counterpart; the present section and the next replace them (Remark~\ref{rem:dictionary}).

\subsection{Apex selection}\label{subsec:apex}

\begin{thm}[Apex selection]\label{thm:apex}
In the situation of Theorem~\ref{thm:pyramid} (so $v\in\mathbb R^n\setminus E$, the toric rays converge along a fixed subsequence to $\psi_t(x)=\varphi(x-tc)+nt$, and $K=\conv(\{p\}\cup B)$ is the resulting pyramid with apex $p$), one has
\[
\langle v,p\rangle=\min_{y\in K}\langle v,y\rangle.
\]
\end{thm}

\begin{proof}
Put $T=n+1$ and $\ell(y)=n-\langle c,y\rangle$ as in Theorem~\ref{thm:pyramid}.

\smallskip\noindent\textbf{Step 1.} In this step, we identify the low-weight index sets as weight-minimizing interpolation sets. For a positive integer $j\leq N_k$, let $A_{k,j}$ be the matrix of the order-less-than-$j$ Taylor map of $H_k$ at the point $1$, written in the monomial basis, so that its kernel is $F_{k,1}^j$. Fix $0<s<T$ and integers $j_k$ with $j_k/k\to s$. For $1\leq j\leq N_k$ put $B_{k,j}=\binom{n+j-1}{n}$. By Theorem~\ref{thm:hypersurface} and Lemma~\ref{lem:jet}, for all large $k$ the matrix $A_{k,j_k}$ has full row rank $B_{k,j_k}$. Let $I_{k,j}\subseteq S_k$ be the coordinate index set of the radial limit of $F_{k,1}^j$ supplied by Proposition~\ref{prop:weights}, and put $J_{k,j}=S_k\setminus I_{k,j}$. Proposition~\ref{prop:weights}(1) gives $J_{k,j}=\{m\in S_k:q_{k,m}<j\}$ and $\#J_{k,j_k}=\operatorname{codim}F_{k,1}^{j_k}=B_{k,j_k}$.

Here and below we use the notation of Lemma~\ref{lem:grassmann} applied with $H=H_k$, $S=S_k$, $e_m=s_{k,m}$, and $F^j=F_{k,1}^j$, so that $e_I=\bigwedge_{m\in I}s_{k,m}$ and $D_R=U_{(e^{Rv_1},\dots,e^{Rv_n})}$ acts by $D_Rs_{k,m}=e^{R\langle v,m\rangle}s_{k,m}$. We claim that $J_{k,j_k}$ minimizes $\sum_{m\in J}\langle v,m\rangle$ among all $B_{k,j_k}$-element subsets $J\subseteq S_k$ on which evaluation of polynomials of degree less than $j_k$ is nonsingular. Indeed, a maximal minor of $A_{k,j_k}$ on a column set $J$ is nonzero exactly when $\ker(A_{k,j_k})\cap\operatorname{span}\{e_m:m\in J\}=0$, which happens exactly when the Pl\"ucker coordinate of $\ker(A_{k,j_k})=F_{k,1}^{j_k}$ on the complementary index set $I=S_k\setminus J$ is nonzero. In the radial degeneration of Lemma~\ref{lem:grassmann}, the limit of $D_RF_{k,1}^{j_k}$ is the coordinate subspace spanned by $\{e_m:m\in I_{k,j_k}\}$. This uses more than the unique-maximizer property of the admissible direction $v$, since a priori several index sets could share the maximal weight sum. Applying $D_R$ to a Pl\"ucker vector $\sum_Ic_Ie_I$ of $F_{k,1}^{j_k}$ scales each coordinate $c_I$ by $e^{R\sum_{m\in I}\langle v,m\rangle}$, so the projectivized Pl\"ucker vector of $D_RF_{k,1}^{j_k}$ converges to the class of the maximal-weight part of $\sum_Ic_Ie_I$; by the continuity of the Pl\"ucker embedding, that class is the Pl\"ucker class of the limit subspace, namely the class of the single coordinate $e_{I_{k,j_k}}$. Consequently the Pl\"ucker coordinate of $F_{k,1}^{j_k}$ on $I_{k,j_k}$ is nonzero, it is the unique nonzero Pl\"ucker coordinate of maximal weight, and its weight $\sum_{m\in I_{k,j_k}}\langle v,m\rangle$ dominates the weight of every other nonzero Pl\"ucker coordinate. Since $\sum_{m\in S_k}\langle v,m\rangle$ does not depend on $J$, the complementary set $J_{k,j_k}$ has minimal $\langle v,\cdot\rangle$-weight, as claimed.

\smallskip\noindent\textbf{Step 2.} In this step, we compute the limit of the normalized weight of $J_{k,j_k}$. Along the fixed subsequence, define again the empirical measures $\Gamma_k=k^{-n}\sum_{m\in S_k}\delta_{(m/k,\,q_{k,m}/k)}$ on $K\times[0,n+1]$ and let $\Gamma$ be any weak subsequential limit. Rerunning the measure-theoretic argument for the fixed data $(v,c)$: the first marginal of $\Gamma$ is Lebesgue measure on $K$, $\int r\,d\Gamma=n\vol(K)$ by Proposition~\ref{prop:weights}(1), and for every $(y,r)\in\operatorname{supp}\Gamma$ with $y\in\intr(K)$, atoms $(m_k/k,q_{k,m_k}/k)\to(y,r)$ together with Proposition~\ref{prop:laplace} and the locally uniform convergence $u_{k,t}\to\psi_t$ give $\psi_t(x)\geq\langle y,x\rangle-\varphi^*(y)+tr$ for all $x,t$; taking Legendre transforms and using the translation formula $\psi_t^*(y)=\varphi^*(y)-t\ell(y)$ yields $r\leq\ell(y)$, and the barycenter identity $\int_K\ell(y)\,dy=n\vol(K)=\int r\,d\Gamma$ upgrades this to $r=\ell(y)$ $\Gamma$-almost everywhere. Since $c\neq0$ by Theorem~\ref{thm:pyramid}, the level sets of $\ell$ are Lebesgue-negligible, so $\Gamma=(\mathrm{id}\times\ell)_*(\mathrm{Leb}_K)$ independently of the subsequential choice, where $\mathrm{Leb}_K$ denotes the unnormalized Lebesgue measure restricted to $K$; being independent of the choice, the weak convergence $\Gamma_k\to\Gamma$ holds along the full fixed subsequence. Since $J_{k,j_k}=\{m:q_{k,m}<j_k\}$, the sum $k^{-n}\sum_{m\in J_{k,j_k}}\langle v,m/k\rangle$ equals $\int_{K\times[0,j_k/k)}\langle v,y\rangle\,d\Gamma_k(y,r)$. Fix $\eta\in(0,s)$. The function $\langle v,y\rangle$ is bounded on $K$ and eventually $s-\eta<j_k/k<s+\eta$, so replacing the moving window $[0,j_k/k)$ by $[0,s-\eta)$ changes the integral by at most a constant multiple of $\Gamma_k\bigl(K\times[s-\eta,s+\eta]\bigr)$, whose upper limit is at most $\vol(\{s-\eta\leq\ell\leq s+\eta\})$ by the Portmanteau theorem; moreover the set $K\times[0,s-\eta)$ has $\Gamma$-negligible boundary, so $\int_{K\times[0,s-\eta)}\langle v,y\rangle\,d\Gamma_k\to\int_{\{\ell<s-\eta\}}\langle v,y\rangle\,dy$. Letting $\eta\downarrow0$ and using the negligibility of the level sets of $\ell$ gives
\[
\frac{1}{k^n}\sum_{m\in J_{k,j_k}}\Bigl\langle v,\frac mk\Bigr\rangle\longrightarrow\int_{C_s}\langle v,y\rangle\,dy.
\]
Since $\#J_{k,j_k}/k^n\to s^n/n!=\vol(C_s)$ by \eqref{eq:capcount}, which holds for the present fixed direction and subsequence, we conclude
\begin{equation}\label{eq:capaverage}
\frac{1}{\#J_{k,j_k}}\sum_{m\in J_{k,j_k}}\Bigl\langle v,\frac mk\Bigr\rangle\longrightarrow\frac{1}{\vol(C_s)}\int_{C_s}\langle v,y\rangle\,dy.
\end{equation}
By the homothety identity $C_s=p+\frac sT(K-p)$ of Theorem~\ref{thm:pyramid}, the barycenter of $C_s$ is $p+\frac sT(0-p)=(1-\frac sT)p$, because the barycenter of $K$ is $0$. Hence the right side of \eqref{eq:capaverage} equals $(1-\frac sT)\langle v,p\rangle$.

\smallskip\noindent\textbf{Step 3.} In this step, we compare with principal lattice interpolation sets near a minimizer and conclude. Fix $\delta>0$ and choose $y_0\in\intr(K)$ with $\langle v,y_0\rangle\leq\min_{z\in K}\langle v,z\rangle+\delta$, and $\rho>0$ with the Euclidean ball of radius $\rho$ about $y_0$ contained in $\intr(K)$. Fix $s$ with $0<s<T$ and $3s<\rho$, take $j_k$ with $j_k/k\to s$, and choose $b_k\in\mathbb Z^n$ with $b_k/k\to y_0$. For all large $k$, the principal lattice set
\[
G_k=\Bigl\{b_k+a:a\in\mathbb Z_{\geq0}^n,\ \sum_ia_i<j_k\Bigr\}
\]
is contained in $S_k$, has $B_{k,j_k}$ elements, and is unisolvent for polynomials of total degree less than $j_k$ by Lemma~\ref{lem:principal}. The minimizing property of $J_{k,j_k}$ from Step~1 therefore gives
\[
\frac{1}{\#J_{k,j_k}}\sum_{m\in J_{k,j_k}}\Bigl\langle v,\frac mk\Bigr\rangle\leq\frac{1}{\#G_k}\sum_{m\in G_k}\Bigl\langle v,\frac mk\Bigr\rangle\leq\langle v,y_0\rangle+C_v\,s+o(1),
\]
where $C_v$ depends only on $v$, because every point of $G_k/k$ lies within $2s+o(1)$ of $y_0$. Combining with Step~2 and letting $k\to\infty$,
\[
\Bigl(1-\frac sT\Bigr)\langle v,p\rangle\leq\min_{z\in K}\langle v,z\rangle+\delta+C_v\,s.
\]
Letting first $s\downarrow0$ at fixed $\delta$ (which keeps the constraint $3s<\rho$ admissible) and then $\delta\downarrow0$ gives $\langle v,p\rangle\leq\min_{z\in K}\langle v,z\rangle$; the reverse inequality holds since $p\in K$.
\end{proof}

\subsection{Dense pyramid apices force a simplex}\label{subsec:densapex}

\begin{thm}\label{thm:densapex}
Let $K\subset\mathbb R^n$ be a full-dimensional compact convex body and let $G\subseteq\mathbb R^n$ be a dense set. Suppose that for every $v\in G$ there exist a point $p_v\in K$, an affine hyperplane $H_v$ not containing $p_v$, and a compact convex set $B_v\subseteq H_v$ such that
\[
K=\conv(\{p_v\}\cup B_v)\qquad\text{and}\qquad\langle v,p_v\rangle=\min_{y\in K}\langle v,y\rangle.
\]
Then $K$ is an $n$-dimensional simplex.
\end{thm}

\begin{proof}
Let $P=\{p_v:v\in G\}$. Every point of $P$ is an extreme point of $K$: if $K=\conv(\{p\}\cup B)$ with $B$ in an affine hyperplane $H$ not containing $p$, choose an affine functional $a$ vanishing on $H$ with $a(p)=1$; then $0\leq a\leq1$ on $K$ and $\{a=1\}\cap K=\{p\}$, so $p$ is exposed, hence extreme.

We claim $K=\overline{\conv}(P)$. Let $w\in\mathbb R^n$ and choose $v_j\in G$ with $v_j\to-w$. Since $p_{v_j}$ minimizes $\langle v_j,\cdot\rangle$ on $K$, we have $\langle v_j,p_{v_j}\rangle=-h_K(-v_j)$, hence
\[
\langle w,p_{v_j}\rangle=h_K(-v_j)+\langle w+v_j,p_{v_j}\rangle\longrightarrow h_K(w),
\]
using continuity of $h_K$ and boundedness of $K$. Thus $\sup_{p\in P}\langle w,p\rangle=h_K(w)$ for every $w$, so $\overline{\conv}(P)$ and $K$ have the same support function and coincide.

Since $K$ is full-dimensional, $P$ affinely spans $\mathbb R^n$; choose affinely independent $p_0,\dots,p_n\in P$. For each $i$, fix a pyramid representation $K=\conv(\{p_i\}\cup B_i)$ with $B_i\subseteq H_i$ and $p_i\notin H_i$. Every extreme point $q$ of $K$ with $q\neq p_i$ lies in $B_i$: any point of $K\setminus(\{p_i\}\cup B_i)$ is of the form $(1-t)p_i+tb$ with $b\in B_i$ and $0<t<1$, hence not extreme. In particular, $H_i\supseteq\{p_j:j\neq i\}$, and since those $n$ points are affinely independent and span an affine hyperplane,
\[
H_i=\aff\{p_j:j\neq i\}.
\]
Suppose some $q\in P$ differs from all of $p_0,\dots,p_n$. Then $q\in H_i$ for every $i$. But in the affine coordinates determined by the affinely independent points $p_0,\dots,p_n$, the hyperplane $H_i$ is the vanishing locus of the $i$-th affine coordinate, and a point on all $n+1$ of them would have all affine coordinates zero, contradicting that affine coordinates sum to one. Hence $P=\{p_0,\dots,p_n\}$ and $K=\conv\{p_0,\dots,p_n\}$, an $n$-dimensional simplex.
\end{proof}

\subsection{The simplex theorem and the reduction}\label{subsec:reduction}

\begin{thm}\label{thm:simplex}
Let $K\subset\mathbb R^n$ be a full-dimensional compact convex body with barycenter $0$, $\intr(K)\cap\mathbb Z^n=\{0\}$, and $\vol(K)=(n+1)^n/n!$. Then $K$ is an $n$-dimensional simplex.
\end{thm}

\begin{proof}
Fix a BB13 potential $\varphi$ of $K$ (Lemma~\ref{lem:potential}), so that the constructions of Sections~\ref{sec:gateway}--\ref{sec:rays} apply. Let $E$ be the countable union of proper hyperplanes of Proposition~\ref{prop:weights} and put $G=\mathbb R^n\setminus E$, a dense subset of $\mathbb R^n$: each hyperplane is Lebesgue-null, so the countable union $E$ is Lebesgue-null and its complement meets every nonempty open set. Fix $v\in G$. Proposition~\ref{prop:weights} produces the integer weights $q_{k,m}$, and Theorem~\ref{thm:limitray} produces, along a subsequence, a toric limit ray $\psi_t$ with $\psi_0=\varphi$, joint convexity in $(x,t)$, and $L(t)=nt$ for all $t$. In the logarithmic coordinates, $V(x,t)=\psi_t(x)$ is a finite convex function on $\mathbb R^n\times[0,\infty)$ whose exponential integrals
\[
Z(t)=\int_{\mathbb R^n}e^{-V(x,t)}\,dx=\int_Xe^{-\psi_t}\,d\nu
\]
are finite and positive with $-\log Z(t)=-\log Z(0)+nt$, by torus invariance and $L(t)=nt$. Proposition~\ref{prop:prekopa} supplies $c_v\in\mathbb R^n$ with $\psi_t(x)=\varphi(x-tc_v)+nt$. Theorem~\ref{thm:pyramid} then makes $K$ a pyramid $K=\conv(\{p_v\}\cup B_v)$ with apex $p_v$ and base $B_v$ contained in the affine hyperplane $H_v=\{y\in\mathbb R^n:\ell_v(y)=n+1\}$ not containing $p_v$, where $\ell_v(y)=n-\langle c_v,y\rangle$ is the affine function of Theorem~\ref{thm:pyramid} associated with $c_v$, and Theorem~\ref{thm:apex} gives $\langle v,p_v\rangle=\min_{y\in K}\langle v,y\rangle$. Thus the hypotheses of Theorem~\ref{thm:densapex} hold for the dense set $G$, and $K$ is a simplex.
\end{proof}

\begin{prop}\label{prop:reduction}
Let $K$ be as in Theorem~\ref{thm:simplex}. Then there is an invertible linear map $A\colon\mathbb R^n\to\mathbb R^n$ with $|\det A|=1$ and $K=AS_n$. Consequently, if every full-rank lattice $\Lambda\subset\mathbb R^n$ of determinant one (that is, of covolume one; the standard lattice notions are recalled at the beginning of Section~\ref{sec:critical}) with $\intr(S_n)\cap\Lambda=\{0\}$ equals $\mathbb Z^n$, then Theorem~\ref{thm:main} holds.
\end{prop}

\begin{proof}
By Theorem~\ref{thm:simplex}, $K$ is a simplex; write its vertices as $p_0,\dots,p_n$. The barycenter of a simplex is the arithmetic mean of its vertices (both the centroid and the vertex mean are affine-equivariant, and for $\Delta_n$ the identity $\int_{\Delta_n}x_i\,dx=\frac{1}{(n+1)!}=\frac{\vol(\Delta_n)}{n+1}$ verifies it directly), so $\sum_ip_i=0$. The vertices $s_0,\dots,s_n$ of $S_n$ (Lemma~\ref{lem:model}) likewise satisfy $\sum_is_i=0$, and $s_i-s_0=(n+1)e_i$ for $1\leq i\leq n$ is a basis of $\mathbb R^n$, as is $p_i-p_0$ by full-dimensionality. Define the linear map $A$ by $A(s_i-s_0)=p_i-p_0$; then
\[
As_0=-\frac{1}{n+1}\sum_{i=1}^nA(s_i-s_0)=-\frac{1}{n+1}\sum_{i=1}^n(p_i-p_0)=p_0,
\]
using $\sum_{i\geq1}p_i=-p_0$, and hence $As_i=p_i$ for every $i$ and $K=AS_n$ as convex hulls of corresponding vertices. Volumes give $\vol(K)=|\det A|\vol(S_n)$, and both volumes equal $(n+1)^n/n!$, so $|\det A|=1$.

For the consequence, put $\Lambda=A^{-1}\mathbb Z^n$, a full-rank lattice with $\det\Lambda=|\det A^{-1}|=1$. Since $A$ is a linear homeomorphism, $\intr(S_n)\cap\Lambda=A^{-1}\bigl(\intr(K)\cap\mathbb Z^n\bigr)=\{0\}$. The assumed critical-lattice statement gives $\Lambda=\mathbb Z^n$, that is, $A^{-1}\mathbb Z^n=\mathbb Z^n$; then the columns of $A^{-1}$ and of $A$ are integral, so $\det A$ and $\det A^{-1}$ are integers with product $1$, whence $A\in\GL_n(\mathbb Z)$ and $K=AS_n=A\bigl((n+1)\Delta_n-(1,\dots,1)\bigr)$, which is the conclusion of Theorem~\ref{thm:main}.
\end{proof}

\section{The critical lattice of the centered simplex}\label{sec:critical}

In this section, we prove the arithmetic half of Theorem~\ref{thm:main}: the standard lattice is the only determinant-one lattice avoiding the interior of the centered standard simplex. The theorem of this section has no established Fano-side counterpart; see the closing paragraph of Remark~\ref{rem:dictionary}. We recall the standard notions from the geometry of numbers (see, e.g., \cite[Sections~3 and~17]{GL87}): a full-rank lattice $\Lambda=M\mathbb Z^n\subset\mathbb R^n$, with $M$ an invertible matrix, has determinant (or covolume) $\det\Lambda=|\det M|$, the volume of a fundamental domain; a lattice \emph{avoids} an open set containing the origin if its only point in that set is the origin; and a lattice of minimal determinant among those avoiding $\intr(S_n)$ is a \emph{critical lattice} of $S_n$.

\begin{lem}\label{lem:cubeavoid}
Let $\Lambda\subset\mathbb R^n$ be a full-rank lattice with $\intr(S_n)\cap\Lambda=\{0\}$. Then $\Lambda$ contains no nonzero point of the open cube $(-1,1)^n$.
\end{lem}

\begin{proof}
Suppose $z\in\Lambda$ is nonzero with $|z_j|<1$ for every $j$. Replacing $z$ by $-z$ if necessary (which changes neither membership in $\Lambda$ nor membership in the open cube), we may assume $\sum_jz_j\leq0$. Then $z_j>-1$ for every $j$ and $\sum_jz_j\leq0<1$, so $z\in\intr(S_n)$ by the halfspace description of Lemma~\ref{lem:model}, contradicting $\intr(S_n)\cap\Lambda=\{0\}$.
\end{proof}

\begin{thm}[Critical-lattice uniqueness]\label{thm:critical}
Let $n$ be a positive integer and let $\Lambda\subset\mathbb R^n$ be a full-rank lattice of determinant one with $\intr(S_n)\cap\Lambda=\{0\}$. Then $\Lambda=\mathbb Z^n$.
\end{thm}

\begin{proof}
By Lemma~\ref{lem:cubeavoid}, $\Lambda$ contains no nonzero point of $(-1,1)^n$. We now invoke Haj\'os's theorem \cite[Satz~32]{Haj41}, which resolved Minkowski's conjecture on lattice cube tilings, in the following form stated in \cite[Theorem~4.2]{BGM+22}: if $\Gamma\subset\mathbb R^n$ is a lattice of covolume one avoiding the open unit cube $(-1,1)^n$, then, after a permutation of the coordinate axes, $\Gamma=U\mathbb Z^n$ for an upper triangular real matrix $U$ all of whose diagonal entries equal one. (Here, as usual in the geometry of numbers, ``avoiding'' refers to the nonzero lattice points; the origin lies in every lattice and in the open cube.) A permutation of coordinates preserves $S_n$, $\intr(S_n)$, $\mathbb Z^n$, and the determinant, so both the hypotheses and the desired conclusion are invariant under it, and we may assume $\Lambda=U\mathbb Z^n$ with $U$ upper unitriangular.

Let $b_k$ denote the $k$-th column of $U$. We prove by induction on $k$ that $e_1,\dots,e_k\in\Lambda$. For $k=1$, upper unitriangularity gives $b_1=e_1$. Assume $e_1,\dots,e_{k-1}\in\Lambda$ and write $b_k=e_k+\sum_{j<k}c_je_j$. Subtracting integer multiples of the available $e_j$ produces a lattice vector
\[
b_k'=e_k+\sum_{j<k}\alpha_je_j\in\Lambda,\qquad0\leq\alpha_j<1.
\]
Suppose some $\alpha_j$ is positive. Put $\sigma=\sum_{j<k}\alpha_j$, let $r$ be the number of positive $\alpha_j$, and put $m=\lfloor\sigma\rfloor+1$; since each positive $\alpha_j$ is strictly less than one, $\sigma<r$, so $1\leq m\leq r$. Choose a set $J$ of $m$ indices $j<k$ with $\alpha_j>0$ and define
\[
x=b_k'-\sum_{j\in J}e_j\in\Lambda.
\]
The $k$-th coordinate of $x$ is $1$; for $j\in J$, the $j$-th coordinate is $\alpha_j-1\in(-1,0)$; for $j<k$ outside $J$, it is $\alpha_j\geq0$; and all later coordinates vanish. Thus every coordinate of $x$ exceeds $-1$ strictly, while
\[
\sum_jx_j=1+\sigma-m<1
\]
because $m>\sigma$. By Lemma~\ref{lem:model}, $x\in\intr(S_n)$, and $x\neq0$ since its $k$-th coordinate is $1$, contradicting $\intr(S_n)\cap\Lambda=\{0\}$. Hence every $\alpha_j$ vanishes and $e_k=b_k'\in\Lambda$, completing the induction.

Therefore $\mathbb Z^n\subseteq\Lambda$; the index of this sublattice pair equals the ratio of the determinants, $[\Lambda:\mathbb Z^n]=\det(\mathbb Z^n)/\det(\Lambda)=1$, so $\Lambda=\mathbb Z^n$. Undoing the coordinate permutation leaves $\mathbb Z^n$ unchanged.
\end{proof}

\begin{rem}\label{rem:criticalcontrol}
By Lemma~\ref{lem:model}, the control lattice $\mathbb Z^n$ does satisfy $\intr(S_n)\cap\mathbb Z^n=\{0\}$, so Theorem~\ref{thm:critical} states precisely that $\mathbb Z^n$ is the unique determinant-one critical lattice of $S_n$. The value $1$ of the critical determinant itself is classical: it is the simplex case of Ehrhart's conjecture, proved in \cite{Ehr79}.
\end{rem}

\section{Proof of the main theorem}\label{sec:proof}

The goal of this section is to prove Theorem~\ref{thm:main}.

\begin{proof}[Proof of Theorem~\ref{thm:main}]
Let $K$ be as in the statement: a full-dimensional compact convex body with barycenter $0$, $\intr(K)\cap\mathbb Z^n=\{0\}$, and $\vol(K)=(n+1)^n/n!$. By Theorem~\ref{thm:simplex}, $K$ is a simplex, and by Proposition~\ref{prop:reduction} there is an invertible linear map $A$ with $|\det A|=1$ and $K=AS_n$; moreover, by the same proposition, it suffices to verify the critical-lattice statement for $S_n$, which is Theorem~\ref{thm:critical}. This gives $A\in\GL_n(\mathbb Z)$ and
\[
K=A\bigl((n+1)\Delta_n-(1,\dots,1)\bigr),
\]
completing the proof with $U=A$. Finally, the classification is sharp: $S_n$ satisfies all three hypotheses by Lemma~\ref{lem:model}, and for every $U\in\GL_n(\mathbb Z)$ so does $US_n$, since a linear map of determinant $\pm1$ preserves volume and fixes the barycenter at the origin, and since $\intr(US_n)\cap\mathbb Z^n=U\bigl(\intr(S_n)\cap\mathbb Z^n\bigr)=\{0\}$ because $U$ is a homeomorphism with $U\mathbb Z^n=\mathbb Z^n$. Hence every member of the classified family is an equality body.
\end{proof}

\end{document}